\documentclass[10pt,leqno]{amsart}
\usepackage{graphicx}
\usepackage{indentfirst,csquotes}

\usepackage{amsmath}
\usepackage{amssymb}
\usepackage{amscd}
\usepackage{latexsym}
\usepackage{amsthm}
\usepackage{lmodern}
\usepackage{lmodern}
\usepackage{fixcmex}
\usepackage{fixcmex}
\usepackage{mathtools}
\usepackage{comment}
\usepackage[cp1251]{inputenc}
\usepackage[refpage]{nomencl}
\usepackage{relsize}
\usepackage[all]{xy}
\usepackage{graphicx}
\usepackage{stackengine}
\usepackage[english]{babel}

\def\today{\number\day\space   \ifcase\month\or January\or
  February\or March\or April\or May\or June\or July\or August\or
  September\or October\or November\or December\fi \space\number\year}

\def\?{{\bf ??}}

\def\Qed{\hbox to 0.5em{ }\nobreak\hfill\hbox{$\square$}}

\let\bs\backslash
\let\del\partial

\def\CC{{\mathbb C}}

\def\NN{{\mathbb N}}

\def\QQ{{\mathbb Q}}

\def\RR{{\mathbb R}}

\def\ZZ{{\mathbb Z}}
\def\w{\omega}

\let\div\relax
\DeclareMathOperator{\div}{div}
\DeclareMathOperator{\Pic}{Pic}

\def\CH{\operatorname{CH}}

\def\div{\operatorname{div}}

\def\I{\operatorname{I}}

\def\Pic{\operatorname{Pic}}

\def\alg{\operatorname{alg}}
\def\Rat{\operatorname{Rat}}
\def\cd{\operatorname{codim}}

\def\reg{\operatorname{reg}}
\def\cal{\mathcal}

\def\dda{\downarrow\kern-0.78em\raise0.25em\hbox{$\downarrow$}}
\def\dua{\uparrow\kern-0.78em\raise0.25em\hbox{$\uparrow$}}

\makeatletter

\let\s@vedleft\left
\let\s@vedright\right
\renewcommand{\left}{\mathopen{}\mathclose\bgroup\s@vedleft}
\renewcommand{\right}{\aftergroup\egroup\s@vedright}
\makeatother

\DeclarePairedDelimiter\abs{\lvert}{\rvert}
\DeclarePairedDelimiter\norm{\Vert}{\Vert}
\DeclarePairedDelimiterX\innerp[2]{\langle}{\rangle}{#1,#2}
\DeclarePairedDelimiterXPP\pnorm[2]{}\lVert\rVert{_{#2}}{#1}

\def \I{{\text{i}}}
\def \bs{{\backslash}}

\def \CC{{\mathbb C}}
\def \QQ{{\mathbb Q}}

\def \ZZ{{\mathbb Z}}

\def \RR{{\mathbb R}}
\def \Dd{{\mathcal D}}

\def\CH{{\text{CH}}}

\def\Totr
\def\del{{\partial}}

\def\ul{\underline}

\def\alg{{\text{alg}}}

\def\Div{{\text{div}}}

\def\I{{\text{i}}}

\newtheorem{thm}{Theorem}[section]
\newtheorem{lem}[thm]{Lemma}

\newtheorem{conj}[thm]{Conjecture}

\newtheorem{prop}[thm]{Proposition}

\newtheorem{cor}[thm]{Corollary}

\newtheorem{defn}[thm]{Definition}
\newtheorem{ex}[thm]{Example}
\newtheorem{rem}[thm]{Remark}

\begin{document}
\date{}
\title{Indecomposable cycles and the case of the twisted Hodge D conjecture}
\author[Karim Mansour]{Karim Mansour}
\email{abdelgal@ualberta.ca}
\maketitle

\begin{abstract}
We will introduce twisted cycles and their associated regulators to cohomology. We prove the conjecture that this regulator is surjective for a general smooth projective surface. We construct indecomposable twisted cycles on elliptic fourfolds, and use that to show that this map is non-zero.
\end{abstract}

\section*{Acknowledgement}
First of all, I would like to thank Kyle Hofmann for checking my ideas for elliptic fourfolds. I would like to give a big thanks to Dr. Totaro for his suggestions.

\begin{section}{Introduction}

\begin{subsection}{Introduction to flat line bundles and twisted cycles}
	In this section, we will introduce twisted cycles. Twisted cycles are "twisted" analogues of higher Chow groups. Higher Chow cycles come from formal sums of divisors of rational functions on subvarieties. Analogously, twisted cycles will be cycles that come from formal sums of divisors of rational sections on a line bundle equipped with a metric over subvarieties. 

\begin{rem}
Twisted cycles quotient out regulator values that share algebraic information across different subvarieties. That is, there is a geometric connection between the subvarieties involved that is encoded precisely as the quotient of the image of the tame symbol. See section 2.
\end{rem}

\begin{rem}
Notice that the reason we add the quotient is in order to remove non-pure regulator values that are shared among subvarieties. That is, removing algebraic redundancies. 
\end{rem}

\begin{rem}
Some of the technical aspects of the introduction can be found in \cite{Lew3}, we just repeat it for the sake of completeness. 
\end{rem}

These line bundles are denoted as flat line bundles. The reason we work with flat line bundles is that they allow us to define regulator maps into cohomology. We define the notion of "flat line bundles."

\begin{defn}
	If $Z$ is smooth, a metric on $L$ is given by a collection of $C^{\infty}$, $\{\rho_{\alpha} : U_{\alpha} \to (0,\infty)\}$  satisfying the following condition:
	
	$$\rho_{\beta} = \rho_{\alpha}|g_{\alpha\beta}|^2 \textrm{ on } U_{\alpha} \cap U_{\beta}$$

where $\{U_{\alpha}\}$ is taken to be a Zariski open
cover of $Z$ where $L|_{U_{\alpha}} \simeq U_{\alpha}\times {\CC}$
trivializes, and where the transition functions of $L$ with respect
to this cover are $\{g_{\alpha\beta}\}$.
\end{defn}

\begin{defn}
	The curvature form
is the global closed real $(1,1)$-form $\nu_L$ given locally on $U_{\alpha}$
by
$$
\nu_L = \frac{1}{2\pi\sqrt{-1}}\partial\overline{\partial}\log \rho_{\alpha},
$$
and it's image $c_1(L) = [\nu_L]$ in $H^2(X,{\ZZ})$ is  the first Chern class of $L$.
\end{defn}

The next lemma gives a corresponding classification to flat line bundles. This classification is more algebraic, allowing us to detect flat line bundles easier.

\begin{lem}
	Let us assume given a line bundle $L$ as above on
a smooth projective $Z$, with first Chern class $c_1(L) = 0$. Then one can
find a corresponding metric for which the curvature form $\nu_L = 0$.
\end{lem}

\begin{proof}
See \cite{GH} (Proposition, p. 148).
From Hodge theory one has $\nu_L \in
\partial\overline{\partial}E^0_Z \implies \nu_L = \frac{1}
{2\pi\sqrt{-1}}\partial\overline{\partial}\psi$ for
some real-valued global $C^{\infty}$ function $\psi$
on $Z$. Set $\rho = \exp(\psi)$, then
$\partial\overline{\partial}\log \rho_{\alpha} = \partial
\overline{\partial}\log \rho$ over each $U_{\alpha}$. Now
replace each $\rho_{\alpha}$ by $\tilde{\rho}_{\alpha} :=
\rho_{\alpha}/\rho$. Then $\tilde{\rho}_{\alpha}$ transforms
accordingly over $U_{\alpha}\cap U_{\beta}$ and hence defines
a metric on $L$; moreover $\{\partial\overline{\partial}
\log \tilde{\rho}_{\alpha}\} = 0$, which is what we needed
to show.
	
\end{proof}

\begin{defn}
A line bundle $L$ on smooth projective variety $Z$ is flat if it's first chern class is zero, i.e. $c_1(L) = 0$. We denote the group of all flat line bundles up to isomorphism by $Pic^0(Z)$. 
\end{defn}

\begin{rem}
The first Chern class map gives a map from holomorphic line bundles to $H^2(Z,\mathbb{Z})$. Flat line bundles are precisely the kernel of that map. We record the following observation: The effective divisors $D_1,D_2 \subset Z$ are homologically equivalent over $Z$, written as $D_1 \sim_{hom,Z} D_2$ implies that there exists a flat line bundle $L$ and non-zero rational section $\sigma$ on $Z$ such that $\div(\sigma)_Z = D_1 - D_2$.
\end{rem}

Next, we will define a flat line bundle over varieties that are not necessarily smooth. Assume that $Z$ is a projective variety, that is not necessarily smooth, and $L$ is a line bundle over $Z$.

\begin{defn}
A metric $||\ ||_L$ on $L$ is given by a collection of $C^{\infty}$ functions, $\{\rho_{\alpha} : U_{\alpha} \to (0,\infty)\}$, $\rho_{\beta} = \rho_{\alpha}|g_{\alpha\beta}|^2 \textrm{ on } U_{\alpha} \cap U_{\beta}$, such that $\rho_{\alpha}$ pullback to $C^{\infty}$ functions with respect to the desingularization $\tilde{Z} \rightarrow Z$ of $Z$. A line bundle is flat if it is flat with respect to the desingularization $\tilde{Z}$. 
\end{defn}

When we combine these two pieces of information, we get the following definition for a flat line bundle over projective variety $Z$:

\begin{defn}
Given a projective variety $Z$. The metric $||\ ||_L$ is flat if $\nu_L = 0$. A flat line bundle $L$ is a pair $(L,||\ ||_L)$ where $||\ ||_L$ is flat. Denote the group of isomorphism classes of flat line bundles $L$ over $Z$ by $Pic^0(Z)$. 
\end{defn}

We are now ready to define the group of twisted cycles. 

\begin{defn}
Assume that $X$ is smooth projective variety. Define the group of twisted cycles $\underline{z}^k(X,1)$, as follows

\begin{align*}
  \underline{z}^k(X,1) := 
  \left\{
    \mathlarger{\sum}_i (\sigma_i, \|\ \|_{L_i}, V_i) : 
    \begin{aligned}
      \operatorname{codim}_X V_i = k -1,\\
      (L_i/V_i, \|\ \|_{L_i}) \text{ \rm{flat}}, \\
      \sigma_i \in \operatorname{Rat}^*(L_i), \\
      \mathlarger{\sum}_i \operatorname{div}(\sigma_i) = 0
    \end{aligned} \ \ \ 
  \right\}
\end{align*}

\end{defn}

\begin{prop}
$\xi = \sum_i(\sigma_i,||\ ||_{L_i})
\otimes Z_i \in \underline{z}^k(X,1)$ and $\w\in E_{X,d-{\rm
closed}}^{n-k+1,n-k+1}$. Then the current defined by
$$
\w \longmapsto \underline{r}(\xi)(\w) := \sum_i\int_{Z_i}\w\log
||\sigma_i||_{L_i},
$$

is $\partial\overline{\partial}$-closed.
	
\end{prop}

\begin{proof}
Let $\w \in \partial\overline{\partial}E_X^{n-k,n-k}$.
Then we can write $\w = d\overline{\partial}\eta$ for some $\eta \in
E_X^{n-k,n-k}$.  Also let 
$$
\xi = \sum_i (\sigma_i,||\ ||_{L_i})\otimes Z_i \in \underline{z}^k(X,1)
$$
 be given as above, and consider the corresponding integral
$$
\sum_i\int_{Z_i}\w\log ||\sigma_i||_{L_i}.
$$
 By Stokes' theorem and a standard calculation (below):
$$
\int_{Z_i}(d\overline{\partial}\eta)\log ||\sigma_i||_{L_i} = \int_{Z_i}
\overline{\partial}\eta\wedge d\log ||\sigma_i||_{L_i}
$$
$$
= \int_{Z_i}
\overline{\partial}\eta\wedge \partial\log ||\sigma_i||_{L_i} = \int_{Z_i}
d\eta\wedge \partial\log ||\sigma_i||_{L_i},
$$
 where the latter two equalities follow by Hodge type, and the former uses
$$
d\big((\overline{\partial}\eta) \log ||\sigma||_L\big) 
= (d\overline{\partial}\eta)
\log ||\sigma||_L - \overline{\partial}\eta\wedge d\log ||\sigma||_L.
$$
More specifically, we make use of the following facts.
By taking $\epsilon$-tubes about the components $D \subset {
\div}(\sigma)$, and using that $\dim D = n-k$ and $\overline{\partial}\eta
\in E_X^{n-k,n-k+1}$, and that $\log ||\sigma||$ is locally $L^1$, in the
limit as $\epsilon \to 0$
$$
\lim_{\epsilon\to 0}\int_{{\rm Tube}_{\epsilon}((\sigma))}
(\overline{\partial}\eta) \log ||\sigma||_L = 0,
$$
(using estimates involving
$\lim_{\epsilon\to 0^{+}}\epsilon\log\epsilon = 0)$.

Note that $\overline{\partial}\eta \wedge \overline{\partial} \log
||\sigma||_L\in E_{X,L^1}^{n-k,n-k+2}$ (locally $L^1$ forms) and that $\dim
Z_i = n-k+1$. Thus we are left with an integral of the form
$\int_Zd\eta\wedge \partial \log ||\sigma||_L$ as indicated above.
Next,
$$
d(\eta\wedge \partial \log ||\sigma||_L) = d\eta\wedge \partial \log
||\sigma||_L,
$$
since  $\overline{\partial}\partial \log ||\sigma||_L = 0$,
by the key ingredient of flatness of $L$. Thus
$$
\int_{Z_i}(d\overline{\partial}\eta)\log ||\sigma_i||_{L_i} =
\int_{Z_i}d(\eta\wedge \partial \log ||\sigma_i||_L) 
$$
$$
= \lim_{\epsilon\to 0}\int_{{\rm Tube}_{\epsilon}((\sigma_i))}
\eta\wedge \partial \log ||\sigma_i||_L.
$$
If we put $Z = Z_i$ for a given $i$, with $z$ a local coordinate
on $Z_{\reg}$, then we have the residue integral:
$$
\lim_{\epsilon\to 0}\int_{|z|=\epsilon}\eta\wedge 
\partial \log |z|^2 = \lim_{\epsilon\to 0}\int_{|z|=
\epsilon}\eta\wedge \frac{dz}{z}
$$
$$
 = 2\pi\sqrt{-1}
\int_{\{z=0\}\cap Z}\eta_{|_{\{z=0\}\cap Z}},
\ \eta_{|_{\{z=0\}\cap Z}} = {\rm Residue}_{\{z=0\}\cap Z}
\big(\eta\wedge \frac{dz}{z}\big),
$$
(i.e. taking ``tubes'' is dual to taking ``residues'').
Then by a residue calculation, linearity,
and Stokes' theorem, we arrive at the formula:
$$
\underline{r}(\xi)(\w) = \frac{-2\pi\sqrt{-1}}{2}
\sum_D\biggl{[}\biggl{(}\sum_i
\nu_D(\sigma_i)\biggr{)}\int_D\eta\biggr{]}.
$$
 (We note that there remains the possibility that $Z = Z_i$ is singular
along $D$.  To remedy this, one may pass to a normalization of $Z$ with the
same calculations above.)  Therefore $\sum_i {\div}(\sigma_i) = 0$,
hence $\underline{r}(\xi)(\w) = 0$ and we are done.
	
\end{proof}

\begin{rem}

{\it Independence of the metric.}\
  Now consider $Z \subset X$ an irreducible subvariety of codimension $k-1$
in $X$, and $L/{Z}$ a flat (algebraic) line bundle given by
$\{g_{\alpha\beta},U_{\alpha}\}$, and $\{\rho_{\alpha}\}$ a flat metric,
i.e., $\partial\overline{\partial}\log \rho_{\alpha} = 0$.  Let $\{
\tilde{\rho}_{\alpha}\}$ be another flat metric. Then
$$
\frac{\tilde{\rho}_{\beta}}{\rho_{\beta}} =
\frac{\tilde{\rho}_{\alpha}\, |g_{\alpha\beta}|^2}
{\rho_{\alpha}\, |g_{\alpha\beta}|^2},
$$
hence $\rho := (\tilde{\rho}_{\alpha} / \rho_{\alpha}) > 0$ 
extends globally over $Z$; moreover, $\partial\overline{\partial}\log \rho =
0$.  It follows that $\lambda := \frac{1}{2}\log \rho 
\in {\RR}$, being harmonic and
global, is necessarily constant.  Therefore
$$
\int_Z\w\log ||\sigma||_{\tilde L} = \int_Z \w\big(\lambda +
\log ||\sigma||_L\big).
$$
But $\lambda\int_Z(\cdots)$ defines a class in $H^{k-1,k-1}_{\rm alg}
(X,{\RR})$.  Thus if
$$
\xi = \sum_j(\sigma_j,||\ ||_{L_j})\otimes Z_j,
$$
then in
$$
\frac{H^{k-1,k-1}(X,{\RR})}{H^{k-1,k-1}_{\rm alg}(X,{\RR})},
$$
the regulator $\underline{r}(\xi)$ is independent of the choice
of flat metric $||\ ||_{L_j}$ on ${L_j}/Z_j$.

\end{rem}

\begin{rem}
What we would like to establish is that twisted cycles have better regulator properties than regular higher-Chow cycles. To establish this we show the following

\begin{itemize}
	\item We show that that any precycles on product of curves can be completed to a twisted cycle (see Remark 1.5).
	\item We establish better regulator properties for a general smooth projective surface.
	\item We establish better regulator properties for elliptic fourfolds..
\end{itemize}
\end{rem}

\begin{rem}
We pin down the intuition for better regulator properties due to the following facts:

\begin{itemize}
	\item For twisted cycles that are points, i.e., twisted cycles of codimension $n - 1$, points can be easily moved across subvarities, and they are homologically equivalent to each other on any subvariety.
	\item On smooth general fourfolds of curves, we have that complete intersections are homologically equivalent to curves that can move across subvarieties in an easy fashion. As complete intersections have good "linear algebraic properties,"  moving them across subvarities is easy. This doesn't occur for all subvarieties. 
\end{itemize}
\end{rem}

\end{subsection}

\section{Some Hodge theory}  
The goal of this section
is to describe the $\partial\overline{\partial}$-closed
regulator current $\underline{r}(\xi)$ given in Proposition 3.3,
from the point of view of de Rham cohomology. A good reference for
this section is \cite{GH} (Chapter 0) and \cite{Sou} (Chapter II).
Let ${\cal D}_{\ell}(X),\ {\cal D}_{r,s}(X)$ be the 
spaces of currents acting on $E_{X}^{\ell}$ and
$E_{X}^{r,s}$ rspectively, and write ${\cal D}^{2n-\ell}(X) = {\cal 
D}_{\ell}(X)$, ${\cal D}^{n-r,n-s}(X) = {\cal D}_{r,s}(X)$. 
One has a corresponding
decomposition 
$$
{\cal D}^{k}(X) = \bigoplus_{p+q=k}{\cal D}^{p,q}(X).
$$

\noindent

\begin{lem}
	
$\partial\overline{\partial}$- \it If  $T\in {\cal 
D}_{X}^{p,q}$ is a coboundary, then $T = 
\partial\overline{\partial}T_{0}$ for some $T_{0} \in {\cal 
D}^{p-1,q-1}(X)$.
\end{lem}

\bigskip\noindent
\begin{cor}
$$
H^{p,q}(X) \simeq \frac{E_{X,d-{\rm closed}}^{p,q}}
{\partial\overline{\partial}E_{X}^{p-1,q-1}} \simeq
\frac{{\cal D}_{d-{\rm closed}}^{p,q}(X)}
{\partial\overline{\partial}{\cal D}^{p-1,q-1}(X)}.
$$
\end{cor}

\noindent

\begin{lem}
	
\ {\it The natural inclusion
$$
{\cal D}^{p,q}_{d-{\rm closed}}(X) \to 
{\cal D}^{p,q}_{\partial\overline{\partial}-{\rm closed}}(X),
$$
induces an isomorphism
$$
\frac{{\cal D}_{d-{\rm closed}}^{p,q}(X)}
{\partial\overline{\partial}{\cal D}^{p-1,q-1}(X)} \simeq
\frac{{\cal D}^{p,q}_{\partial\overline{\partial}-{\rm closed}}(X)}
{\partial {\cal D}^{p-1,q}(X) + \overline{\partial}{\cal D}^{p,q-1}(X)}.
$$}

\noindent
\end{lem}

{\it Proof.}\ Let $T\in {\cal D}_{\partial\overline{\partial}-{\rm
closed}}^{p,q}(X)$. Then $\partial T\in {\cal D}^{p+1,q}(X)$ and
$\overline{\partial}T\in {\cal D}^{p,q+1}(X)$ are both $d$-closed.
Therefore, from Hodge theory, $\partial T = dS_{1}$ and
$\overline{\partial}T = dS_{2}$ for some $S_{1}\in F^{p+1}{\cal 
D}^{p+q}(X)$, and $S_{2}\in \overline{F^{q+1}{\cal D}^{p+q}(X)}$.
Thus $d(T-S_{1} - S_{2}) = 0$ and moreover by the Hodge $(p,q)$
decomposition theorem, we can modify $S_{j}$
within it's Hodge  type, such that
the cohomology class $[T-S_{1} - S_{2}]$ is of type $(p,q)$. 
(More explicitly, we can write
$$
[T-S_{1}-S_{2}] = [A_{1}] \oplus [B]\oplus [A_{2}],
$$
where 
$$
[A_{1}]\in F^{p+1}H^{p+q}(X,{\CC}),\quad [A_{2}]
\in \overline{F^{q+1}H^{p+q}(X,{\CC})},
\quad [B]\in H^{p,q}(X),
$$
are represented by $d$-closed currents (or forms) $A_{1}$, $A_{2}$,
$B$ of the corresponding Hodge types.
Now replace $S_{j}$ by $S_{j}-A_{j}$, and relabel
it $S_{j}$.) Hence
there exists $T_{0}$ such that $T - S_{1} - S_{2} + dT_{0}\in
{\cal D}^{p,q}_{d-{\rm closed}}(X)$. This implies that
$T + \partial T_{0}^{p-1,q} + \overline{\partial}T_{0}^{p,q-1}$ is
$d$-closed. Next, suppose that $T\in {\cal D}^{p,q}_{d-{\rm 
closed}}(X)$ is given such that $T \in {\rm Im}\partial + 
{\rm Im}\overline{\partial}$. By the Hodge theorem, $T$ has no harmonic
part, and being $d$-closed implies that it is a coboundary. The
lemma easily follows from this.
\qed

\bigskip
To arrive at the same sort of de Rham description
of $\underline{r}(\xi)$ that appears in the twisted case
in \cite{Ja} (p. 349), we for the moment include
the twist factor $1/(2\pi\sqrt{-1})^{n-k+1}$ appearing in (3.0).
It is obvious that $\underline{r}(\xi)$ given in Proposition 3.3 determines
an element of ${\cal D}_{n-k+1,n-k+1,\partial\overline{\partial}-
{\rm closed}}(X,{\RR}(n-k+1))$. It follows easily from the proof
of Lemma 4.3, that there exists $\psi\in {\cal D}_{2n-2k+2,0}(X)\oplus
\cdots \oplus {\cal D}_{n-k+2,n-k}(X)$ such that
$\underline{r}(\xi) + \pi_{n-k+1}(\psi)$ is $d$-closed. It's
action on 
$$
E_{X,d-{\rm closed}}^{n-k+1,n-k+1}/\partial
\overline{\partial}E_{X}^{n-k,n-k}
$$ 
is the same as $\underline{r}(\xi)$.
By duality, viz.,
$$
H^{k-1,k-1}(X,{\RR}(k-1)) \simeq H^{n-k+1,n-k+1}(X,{\bf 
R}(n-k+1))^{\vee},
$$
we end up with a class $\underline{r}(\xi) \in H^{k-1,k-1}(X,
{\RR}(k-1))$. Note that likewise
$$
[\underline{r}(\xi)] \in \frac{{\cal D}^{k-1,k-1}_{\partial
\overline{\partial}-{\rm closed}}(X)}
{\partial {\cal D}^{k-2,k-1}(X) + \overline{\partial}{\cal D}^{k-1,k-2}(X)}
\simeq H^{k-1,k-1}(X).
$$
Let
$$
Q^{k-1,k-1} = \frac{{\cal D}^{k-1,k-1}_{\partial
\overline{\partial}-{\rm closed}}
(X)}{{\cal D}^{k-1,k-1}_{X,d-{\rm closed}}}.
$$
There is a commutative diagram of short exact sequences:
$$
\begin{matrix}&&0&&0&&0\\
&\\
&&\downarrow&&\downarrow&&\downarrow\\
&\\
0&\to&{\rm Im}\partial\overline{\partial}&\to&{\rm Im}\partial 
+ {\rm Im}\overline{\partial}&\to&\frac{{\rm Im}\partial + {\rm
Im}\overline{\partial}}{{\rm Im}\partial\overline{\partial}}&\to&0\\
&\\
&&\downarrow&&\downarrow&&\downarrow\\
0&\to&{\cal D}^{k-1,k-1}_{X,d-{\rm closed}}&\to&
{\cal D}^{k-1,k-1}_{\partial\overline{\partial}-{\rm
closed}}&\to&Q^{k-1,k-1}&\to&0\\
&\\
&&\downarrow&&\downarrow&&\downarrow\\
&\\
0&\to&H^{k-1,k-1}(X)&=&H^{k-1,k-1}(X)&\to&0\\
&\\
&&\downarrow&&\downarrow\\
&\\
&&0&&0\end{matrix}
$$
Thus
$$
Q^{k-1,k-1} \simeq \frac{{\rm Im}\partial + 
{\rm Im}\overline{\partial}}{{\rm Im}
\partial\overline{\partial}}.
$$
Note that $\underline{r}(\xi)$ determines a class 
$$
\{\underline{r}(\xi)\}\in Q^{k-1,k-1},
$$
which is a measure of how far $\underline{r}(\xi)$ is from being
$d$-closed. 

\bigskip\noindent
\begin{prop}
{\it Let $X$  be a projective algebraic manifold of 
dimension $n$, and $D$ an algebraic cycle of dimension
$k-1$ on $X$. Next,
let $\xi \in \underline{z}^{k}(X,1)\otimes {\RR}$ be given
and consider the corresponding $\underline{r}(\xi)$. Let
us write $\underline{r}(\xi) = \sum_{i,\alpha}r_{i}\int_{Z_{i,\alpha}}
\log||\sigma_{i,\alpha}||\wedge (?)$, with $r_{i}\in {\RR}$, 
and assume that $D$ meets each $Z_{i,\alpha}$ 
properly (i.e. in a $0$-dimensional set), and
that $|D| \cap |(\sigma_{i,\alpha})_{Z_{i,\alpha}}| = \emptyset$.
Then if we put $[D]$ to be the Poincar\'e dual of $D$, the
cup product is given by the formula:
$$
\langle \underline{r}(\xi),[D]\rangle = \sum_{i,\alpha}r_{i}
\int_{Z_{i,\alpha}\cap D}\log||\sigma_{i,\alpha}||.
$$}
\end{prop}

\noindent
{\it Proof.}\ By desingularization and
linearity, we reduce to the case where $j : D\hookrightarrow X$
is a smooth subvariety of $X$. Let $[\gamma]$ the Poincar\'e dual of 
any given cycle $\gamma$ on $X$. Then
$$
j_\ast\circ j^\ast[\gamma] = [\gamma\cap D].
$$
This follows from
$$
j_{\ast}\circ j^{\ast}[\gamma] = \langle j_\ast\circ j^\ast[\gamma],
[X]\rangle = j_\ast\langle j^\ast[\gamma],j^\ast[X]\rangle
$$
$$
= j_\ast\langle j^\ast[\gamma],[D]\rangle
= \langle [\gamma],j_\ast[D]\rangle_X = [\gamma\cap D].
$$
In this case $\underline{r}(\xi) \in H^{k-1,k-1}(X)$ 
has a well-defined pullback $j^\ast \underline{r}(\xi)
\in H^{k-1,k-1}(D)$, where $\dim_XD = k-1$, and where 
in this case, $j^\ast \underline{r}(\xi) =
\sum_\alpha \int_{Z_\alpha \cap D}
\log||\sigma_\alpha||$. Note that
$j_\ast$ is just the trace. 
The proposition follows from this.
\qed

\bigskip\noindent
\begin{rem}
(i) It is easy to show that
$\underline{r}(\xi)$ is $d$-closed $\Leftrightarrow$ it
is a ${\RR}$ combination of algebraic cycles. 
This is generalized in Theorem 6.10 below.

\bigskip
(ii) The formula in  Proposition $2.4$ can be interpreted in terms
of height pairings. Let us further assume that
$|D| \cap \big(\bigcup_{i,\alpha}Z_{i,\alpha,{\rm Sing}}\big) = 
\emptyset$. Then for a suitable choice of flat metrics, we
have:
$$
\langle \underline{r}(\xi),[D]\rangle = \sum_{i,\alpha}r_{i}
\langle (\sigma_{i,\alpha})_{\tilde{Z}_{i,\alpha}}, D\cap 
\tilde{Z}_{i,\alpha}\rangle_{\rm ht},
$$
where $\langle\ ,\ \rangle_{\rm ht}$ is the height pairing
on a desingularization $\tilde{Z}$. The version of the definition
of height pairing we employ is given in \cite{MS3} (Def. 1).
(One need only show that
for a suitable flat metric, ${\bf H}(\log||\sigma||) = 0$,
where ${\bf H}(-)$ is the harmonic projection. However if we
write $c = {\bf H}(\log||\sigma||)$, then we know $c\in {\RR}$ is
constant. Put $\lambda = {\rm e}^{-c} > 0$, and multiply the metric
$\rho$ by $\lambda\cdot \rho$.)
\end{rem}

\section{A Tame symbol}
Now let $Z \subset X$ be of codimension $k-2$ with given flat
bundles $(L_j,||\ ||_{L_j})$, $j = 1,2$ and $\sigma_f \in 
{\Rat}^{\ast}(L_1),\ \sigma_g\in {\Rat}^{\ast}(L_2)$.
As a first step in the direction of constructing a twisted
Milnor complex, we define a
generalization of the Tame symbol as follows.
$$
\underline{T}(\{(\sigma_f, ||\ ||_{L_1}),(\sigma_g,
||\ ||_{L_2})\}\otimes Z)
$$
$$
= \sum_{\cd_ZD=1}\biggl{(}(-1)^{\nu_D(\sigma_f)\nu_D(\sigma_g)}
\biggl{(}\frac{\sigma_f^{\nu_D(\sigma_g)}}{\sigma_g^{\nu_D(\sigma_f)}}
\biggr{)}_D,||\ ||_{L_1^{\otimes
\nu_D(\sigma_g)}\otimes L_2^{-\nu_D(\sigma_f)}}\biggr{)}\otimes D .
$$

\noindent
\begin{prop}

$\underline{T}(\{(\sigma_f, ||\ ||_{L_1}),(\sigma_g, ||\
||_{L_2})\}\otimes Z) \subset \underline{z}^k(X,1)$.
\end{prop}

\bigskip
\noindent
{\it Proof.}\  One first shows that $\underline{T}$ takes flat line bundles
over $Z$ to flat line bundles over each such $D \subset Z$ of codimension
one. If $\{f_{\alpha\beta}\}$, resp.\space $\{g_{\alpha\beta}\}$ are
the transition functions for $L_1$, resp. $L_2$,  over $Z$ with common
open trivializing cover $\{U_{\alpha}\}_{\alpha\in I}$, and with
corresponding $\{\rho_{\alpha}^{(j)} : U_{\alpha} \to (0,\infty)\}$, $j =
1, 2$ i.e. $\partial\overline{\partial} \rho_{\alpha}^{(j)} = 0$, we
consider the following calculation for codimension one irreducible $D
\subset Z$.  First, for $\sigma_f = \{f_{\alpha}\}$ and $\sigma_g =
\{g_{\alpha}\}$ local representations of nonzero meromorphic
sections of $L_1$ and $L_2$ we have  $f_{\alpha} = f_{\alpha\beta}
f_{\beta}$ and $g_{\alpha} = g_{\alpha\beta}g_{\beta}$ over $U_{\alpha}\cap
U_{\beta}$.  Hence over $D \cap U_{\alpha}\cap U_{\beta}$,
$$
\frac{f_{\alpha}^{\nu_D(g)}}{g_{\alpha}^{\nu_D(f)}} = h_{\alpha\beta}
\frac{f_{\beta}^{\nu_D(g)}}{g_{\beta}^{\nu_D(f)}},
$$
where $h_{\alpha\beta} = f_{\alpha\beta}^{\nu_D(\sigma_g)}
g_{\alpha\beta}^{-\nu_D(\sigma_f)}$ on $D \cap U_{\alpha}\cap U_{\beta}$.
Further, we have the metric $\{(\rho_{\alpha}^{(1)})^{\nu_D(\sigma_{g})}
(\rho_{\alpha}^{(2)})^{{-\nu_D(\sigma_{f})}}\}$ 
associated to the line bundle
$\{h_{\alpha\beta}\}$ over $D$, and with respect to the open cover
$\{U_{\alpha}\cap D\}_{\alpha\in I}$. But
$$
\partial\overline{\partial}\log 
((\rho_{\alpha}^{(1)})^{\nu_D(\sigma_{g})}
(\rho_{\alpha}^{(2)})^{{-\nu_D(\sigma_{f})}})\hskip2in
$$
$$
= \nu_D(\sigma_g)\partial\overline{\partial}\log (\rho_{\alpha}^{(1)})
- \nu_D(\sigma_f)\partial\overline{\partial}\log (\rho_{\alpha}^{(2)}) = 0,
$$
hence this metric is flat as well.
Next, we show that the divisor associated to $\underline{T}(\{(\sigma_f,
||\ ||_{L_1}),(\sigma_g, ||\ ||_{L_2})\}\otimes Z)$ is zero.  Choose a
Zariski open set $U \subset Z$ for which $L_1,\ L_2$ both trivialize
over $U$.  Then over $U$ we have
(the restriction of) the divisor associated to the usual tame symbol,
$T(\{f,g\}\otimes Z)$, which is zero, as required.  Here we have identified
 $f =\sigma_f$ and $g = \sigma_g$ for $f, g\in {\CC}(Z)^{\times}$. 
Since $Z$ is covered by such $U$, we are done.
\qed

\bigskip\noindent

\begin{prop}
$\underline{r}({\rm Im}(\underline{T})) = 0$.
\end{prop}

\bigskip
We prove this by first establishing two Lemmas. 

\bigskip
\noindent
\begin{lem}
{\it Let $Z$ be a smooth subvariety of codimension
$k-2$ in $X$, and let $f,g \in {\CC}(Z)^{\times}$ be given.
Then $\underline{r}\big(T\{f,g\}\big) = 0$.}
\end{lem}
\bigskip
\noindent
{\it Proof.}\ (See \cite{Lev}). It is instructive
to sketch the proof. By a pushforward of the relevant currents involved,
and by a proper modification, it suffices to assume that $Z$
is smooth and that $f, g : Z\to {\bf P}^{1}$ are morphisms.
Put $F = (f,g) : Z\to {\bf P}^{1}\times {\bf P}^{1}$, and let
$(t,s) = (z_{1}/z_{0},w_{1}/w_{0})$ be affine coordinates
for ${\bf P}^{1}\times {\bf P}^{1}$. Then $F^{\ast}T\{t,s\}
= T\{f,g\}$. One can explicitly compute:
$$
T\{t,s\} = (\infty\times {\bf P}^{1},s) - (0\times {\bf P}^{1},s)
+ ({\bf P}^{1}\times 0,t) - ({\bf P}^{1}\times \infty,t).
$$
Next, put
$$
\eta = \int_{\infty\times {\bf P}^{1}}\log|s|\wedge ? - 
\int_{0\times {\bf P}^{1}}\log|s| \wedge ? + 
\int_{{\bf P}^{1}\times 0}\log|t| \wedge ?
- \int_{{\bf P}^{1}\times \infty}\log|t|\wedge ?.
$$
It is easy to see that $F^{\ast}\eta = \underline{r}(T\{f,g\})$,
and that $\eta$ defines the zero cohomology class.
This proves the lemma. \qed

\bigskip
Now let $\sigma$ be a  section of a flat line
bundle over a given subvariety $Z\subset X$. Then  $||\sigma||$
has at worst pole like growth
along the pole set of $\sigma$; moreover
$\partial\overline{\partial} \log||\sigma|| = 0$.

\bigskip\noindent
\begin{lem}
{\it In terms of local analytic coordinates, 
$||\sigma||$ is locally a product of the
form $\rho = h\overline{h}$, where $h$ is meromorphic.}

\end{lem}
\bigskip
\noindent
{\it Proof.}\ Since $\overline{\partial}\partial \log||\sigma|| = 0$,
it follows that ${\partial \log ||\sigma||}$
is a meromorphic $1$-form. Therefore by the
holomorphic Poincar\'e lemma, and away from divisor set $|(\sigma)|$, 
locally (in the strong topology) we have
$$
{\partial \log ||\sigma||} = {\partial H},
$$
for some holomorphic function $H$. Therefore locally
$$
d\log ||\sigma|| = {\partial H} + \overline{\partial}\ \overline{H}
= (\partial+\overline{\partial})H + (\partial+\overline{\partial})
\overline{H} = d(H+\overline{H}).
$$
Thus $\log||\sigma|| = H+\overline{H} + K$ for some $K\in {\RR}$,
and hence $\rho := h\overline{h}$ where $h= {\rm e}^{H + \frac{K}{2}}$.
Next, locally over the divisor set $|(\sigma)|$, we
can replace $\rho$ by  $\tilde{\rho} = |f|^{2}\rho$, for some
meromorphic function $f$, such that $\log \tilde{\rho}$ is defined
and $\partial\overline{\partial}\log \tilde{\rho} = 0$. Since
$\tilde{\rho}$ is a product (local) of a holomorphic function
times it's conjugate, it follows that $\rho$ is a product of
a meromorphic function times it's conjugate.
\qed

\bigskip
\noindent
{\it Proof of Proposition $3.2.$}\
Observe that if $h\overline{h} = k\overline{k}$
on an open set in ${\CC}^{n}$, with $h$, $k$
meromorphic, then 
$$
\frac{h}{k} = \overline{\biggl(\frac{h}{k}\biggr)}^{-1},
$$
is both $\partial$ and $\overline{\partial}$-closed. Thus
$h = ck$ for some $c\in {\CC}^{\times}$ with $|c|=1$.
Thus over an analytic cover $\{\Delta_{\alpha}\}$ of $Z$,
we can write $||\sigma||\big|_{\Delta_{\alpha}} = 
h_{\alpha}\overline{h}_{\alpha}$, where
$h_{\alpha}$ is a meromorphic function on $\Delta_{\alpha}$. 
Suppose that there is a finite cover $\tilde{Z}\to Z$ such that 
by analytic continuation, $h$ becomes a rational function on
$\tilde{Z}$. Let's write this as $\tilde{h}$. Then
$||\tilde{\sigma}|| = \tilde{h}\overline{\tilde h}$, where 
$\tilde{\sigma}$ is the pullback of $\sigma$ to $\tilde{Z}$. 
This for example
would be the case if the $c$'s above are $m$-th roots
of unity for some $m\in {\NN}$. Arriving at this
situation would imply Proposition 3.2, by putting us in
the setting of Lemma 3.4. However
by a limit argument, we can reduce to this situation.
Put $S^{1} = \{z\in {\CC}\ \big|\ |z|=1\}$. 
Then $\{h_{\alpha}\}$ naturally defines an element in
$H^{0}(Z,{\cal M}_{Z}^{\times}/S^{1})$. We can assume that $Z$
is smooth (and projective). The map $H^{1}(Z,S^{1})
\to H^{1}(Z,{\cal M}_{Z}^{\times}) = 0$ factors through
$H^{1}(Z,{\cal O}_{Z}^{\times}) \to H^{1}(Z,{\cal M}_{Z}^{\times})$,
which is well-known to be zero (see \cite{GH}). 
From the short exact sequence 
$$
0\to S^{1} \to {\cal M}_{Z}^{\times} \to {\cal M}_{Z}^{\times}/S^{1}
\to 0,
$$
we deduce that 
$$
H^{0}(Z,{\cal M}_{Z}^{\times}/S^{1}) \to H^{1}(Z,S^{1}) = H^{1}(Z,
{\RR}/{\ZZ}),
$$
is surjective, where ${\rm e}^{\sqrt{-1}t} : {\RR}/{\ZZ}
\ {\buildrel \sim\over\to}\ S^{1}$. Since the kernel
of the map $H^{2}(Z,{\ZZ})\to H^{2}(X,{\QQ})$ is a finite
group, there is no loss of generality in identifying
$H^{1}(Z,{\RR}/{\ZZ})$ with $H^{1}(Z,{\RR})/H^{1}(Z,{\ZZ})$,
and $H^{1}(Z,{\QQ}/{\ZZ})$ with 
$H^{1}(Z,{\QQ})/H^{1}(Z,{\ZZ})$.
Next, since $||\sigma||\big|_{\Delta_{\alpha}} =
h_{\alpha}\overline{h}_{\alpha}$, it follows that 
${\div}(\{h_{\alpha}\}) = {\div}(\sigma)$, and hence the class of
$\{h_{\alpha}\}$ in $H^{1}(Z,{\RR}/{\ZZ})$, which we
identify with $H^{1}(Z,{\RR})/H^{1}(Z,{\ZZ}) \simeq 
H^{0,1}(Z)/H^{1}(Z,{\ZZ}) \simeq {\Pic}^{0}(Z)$, 
is the class corresponding to the flat line bundle associated
to $\sigma$. Let ${\cal L}\big/{\Pic}^{0}(Z)\times Z$ be
the Poincar\'e bundle, and let $\tilde{\sigma}$ be a
rational section of ${\cal L}\big/{\Pic}^{0}(Z)\times Z$
which doesn't vanish on $\{0\}\times Z$.\footnote{This
is easy to arrange. With respect to a projective embedding
of ${\Pic}^{0}(Z)\times Z$, the twisted bundle ${\cal L}(m)$
is very ample for $m>>1$. One simply chooses a general section
of $\Gamma({\cal L}(m))$ and of $\Gamma({\cal O}(m))$
and assigns $\sigma$ to be the quotient.} Then for
$t$ in some polydisk neighbourhood of $0\in {\Pic}^{0}(Z)$,
one has a family of flat metrics $||\ ||_{t}$ on $L_{t} :=
{\cal L}\big|_{\{t\}\times Z}$, and if we put $\sigma_{t} =
\tilde{\sigma}\big|_{\{t\}\times Z}$, then locally
$||\sigma_{t}||_{t} = h_{t,\alpha}\overline{h}_{t,\alpha}$,
and we arrive at a deformation $\{h_{t,\alpha}\}$ of
$\{h_{0,\alpha}\}$.  Next, by our identifications above,
$H^{1}(Z,{\QQ}/{\ZZ})$ is a dense subset
of the torus $H^{1}(Z,{\RR}/{\ZZ})$, and a point in 
$H^{1}(Z,{\QQ}/{\ZZ})$ corresponds to a class 
$\{\underline{h}_{\alpha}\}$ with corresponding 
$c$'s being $m$-th roots of unity.
Thus with regard to $||\sigma||$,
we can write $\{h_{\alpha}\}$ as a limit of classes
corresponding to points in $H^{1}(Z,{\QQ}/{\ZZ})$.
Namely,
$$
\big\{\underline{\underline{h}}_{t,\alpha} := \big(h_{t,\alpha}
\cdot h^{-1}_{0,\alpha}\big)\cdot h_{\alpha}\big\} 
\quad {\buildrel {t\to 0}
\over \mapsto}\quad \{h_{\alpha}\},
$$
where $\{\underline{\underline{h}}_{t,\alpha}\}$ corresponds
to a class in $H^{1}(Z,{\QQ}/{\ZZ})$ for $t$ in a countably dense
subset of some polydisk neighbourhood of $0\in {\Pic}^{0}(Z)$.
This proves the proposition.
\qed

\section{A Milnor complex}
We first recall the definition of Milnor
$K$-theory \cite{BT}. Let ${\bf F}$ be a field 
with multiplicative group ${\bf F}^{\times}$, and set
$$
 T({\bf F}^{\times}) = \bigoplus_{n\geq 0}T^n({\bf F}^{\times}),
$$
the tensor algebra of the ${\ZZ}$-module ${\bf F}^{\times}$. 
(Here $T^{0}({\bf F}^{\times}) := {\ZZ}$.) Then
${\bf F}^{\times} \ {\buildrel\sim\over \to}\ T^1({\bf F}^{\times})$ by
$a\mapsto [a]$. If $a \ne 0,\ 1$, set $r_a = [a]\otimes [1-a]$ in $T^2({\bf
F}^{\times})$. Then the two-sided ideal $R$ generated by the $r_a$'s is
graded, and we put
$$
 K_{\ast}^M{\bf F} = T({\bf F}^{\times})\big/R = 
\bigoplus_{n\geq 0}K^{M}_n{\bf F}.
$$
 Then $K_{\ast}^M{\bf F}$ may be presented as a ring with generators
$\ell(a)$, for $a\in {\bf F}^{\times}$, subject to the relations
$$
\ell(ab) = \ell(a) + \ell(b), 
$$
$$
\ell(a)\ell(1-a) = 0,\quad a \ne 0,\ 1.
$$
Observe that $K_j^M({\bf F}) = K_j({\bf F})$ for $j = 0,1,2$,
where the latter is Quillen $K$-theory.
 
\bigskip\noindent
{\it 4.1. A twisted Milnor complex.} \
As before let $L$ be a (flat) line
bundle over $Z\subset X$. We view $L$ in terms of a corresponding
Cartier divisor, viz., work on the Cartier divisor level.
Then we first observe that although the set of nonzero rational 
sections ${\Rat}^{\ast}(L/Z)$ of $L$ over $Z$
is not a group, by fixing $Z$,
the union $\coprod_{L/Z}{\Rat}^{\ast}(L/Z)$ can be
endowed with the structure of a group. More specifically,
an element of $\coprod_{L/Z}{\Rat}^{\ast}(L/Z)$ is given 
by a pair $(\sigma,L)$, $\sigma \in 
{\Rat}^{\ast}(L/Z)$, $L/Z$ flat, with product
structure $(\sigma_1,L_1)\star (\sigma_2,L_2) = (\sigma_1\sigma_2,
L_1\otimes L_2)$. Thus we assign
$$
\underline{K}_{1,L/Z}^M({\CC}(Z)) := {\Rat}^{\ast}(L/Z).
$$
In terms of local trivializations, 
$$
\sigma \in {\Rat}^{\ast}(L/Z) \Leftrightarrow 
\big\{\sigma = \{\sigma_{\alpha}\} \mid \sigma_{\alpha} \in K_1(
{\CC}(Z))\ {\rm and} \ g_{\alpha\beta}\sigma_{\beta} = 
\sigma_{\alpha}\big\},
\leqno{(4.2)}
$$
where $\{g_{\alpha\beta}\}$ defines $L$, and 
where $\{\sigma_{\alpha}\}$ lies in
the direct product $\prod_{\alpha}K_1({\CC}(Z))$.  By passing to a direct
limit over refining open covers of $Z$,
the latter term in (4.2), which will still be denoted by
$\underline{K}_{1,L/Z}^M({\CC}(Z))$, is given the
structure we require, and we will assume this to be the case in the
discussion below.  Next, we consider the group
$$
{\underline{K}}^M_{1,Z} = \coprod_{L/Z\ {\rm flat}}
\underline{K}_{1,L/Z}^M({\CC}(Z)) = \coprod_{L/Z\ {\rm flat}}
{\Rat}^{\ast}(L/Z).
$$ 
 
Put
$$
\underline{K}^M_{\bullet,Z}\ =\ \biggl(\sum_{j=0}^{\infty}
\big({\underline{K}}^M_{1,Z}\big)^{\otimes_{\bf Z}j}\biggr)
\biggl/ R^{\bullet},
$$
where $R^{\bullet}$ is the two-sided ideal generated by

\bigskip
1)\ $\big\{\sigma \otimes (-\sigma)\ \big|\ \sigma \in 
\underline{K}^M_{1,Z}\big\}$,

\bigskip
2)\ $\big\{ (1-f)\otimes f\ \big|\ f\in {\CC}(Z)^{\times}\backslash 
\{1\}\big\}$.

\bigskip
The first relation leads to a desired anticommutative property
of products of ``symbols'', and the second relation incorporates
the ``usual'' Steinberg relation in the case of function fields.
Since $R^{\bullet} = \oplus_{j\geq 2}R^{j}$ is graded, we can write 
$$
\underline{K}^M_{\bullet,Z} = \bigoplus_{j=0}^{\infty}
\underline{K}^M_{j,Z},
$$
where for example ${\underline{K}}^M_{0,Z} = {\ZZ}$,
$\underline{K}^M_{1,Z}$ is given above,
and $\underline{K}^M_{2,Z} =$
$$
\big\{{\rm group\ of\ symbols}\ 
\{\sigma_{1},\sigma_{2}\}\ \big/\ {\rm [generalized]\ Steinberg
\ relations}\big\}.
$$
Thus for example, $\{\sigma,-\sigma\} = 1$, hence
one can easily check that
$$
\{\sigma,\sigma\} = \{\sigma,-1\} = \{-1,\sigma\} = 
\{\sigma,\sigma^{-1}\}.
$$
In general, given a symbol $\{\sigma_{1},\sigma_{2}\}$,
and up to rewriting this as a product of other symbols, 
one can always ``factor out'' common divisors
in the divisor sets $|(\sigma_{1})|,\ |(\sigma_{2})|$.

\bigskip
\noindent
4.3.\quad We want to build a twisted Milnor-Gersten complex
out of this, with the $j$-th term given by 
$$
\bigoplus_{\cd_XZ = k-j}\underline{K}_{j,Z}^M.
$$

Thus
$$
\bigoplus_{{\rm codim}_XZ = k} \underline{K}_{0,Z}^M
= \bigoplus_{\cd_XZ = k} K_0({\CC}(Z)),
$$
and the first two (generalized) Tame symbols
$$
{\div} := T^{(1)}: \bigoplus_{{\rm codim}_XZ = k-1}\underline{K}_{1,Z}^M
\to \bigoplus_{\cd_XZ = k}\underline{K}_{0,Z}^M,
$$
$$
T := T^{(2)}: \bigoplus_{\cd_XZ = k-2}\underline{K}_{2,Z}^M
\to \bigoplus_{\cd_XZ = k-1}\underline{K}_{1,Z}^M,
$$
have already been defined.  In order to define higher Tame symbols
$$
T^{(j+1)}: \bigoplus_{\cd_XZ = k-j-1}\underline{K}_{j+1,Z}^M
\longrightarrow \bigoplus_{\cd_XZ = k-j}\underline{K}_{j,Z}^M,
$$
we  digress again \cite{BT}.  This time we 
will assume given a field ${\bf F}$
with a discrete valuation $\nu : {\bf F}^{\times} \to {\ZZ}$, and the
corresponding discrete valuation ring ${\cal O} := \{a \in {\bf F} \ |\
\nu(a) \geq 0\}$, where we assign $\nu (0) = \infty$.  Let $\pi\in {\cal O}$
generate the unique maximal ideal $(\pi)$, i.e., $\nu(\pi) = 1$, and recall
that all other nonzero ideals are of the form $(\pi^m)$, for $m \geq 0$. 
Note that ${\bf F}^{\times} = {\cal O}^{\times}\cdot \pi^{\ZZ}$ (direct
product).  Next let ${\bf k} = {\bf k}(\nu)$ be the residue field, and
$K_{\bullet}^M$ Milnor $K$-theory.  Then there is
a map
$$
d_{\pi} : {\bf F}^{\times} \to (K^{M}_{\bullet}{\bf k}(\nu))(\Pi) :
u\pi^i \mapsto \ell(\overline{u}) + i\Pi,
$$
where $\Pi = \ell(\pi)$, and satisfies $\Pi^2 = \ell(-1)\Pi$. This map
induces
$$
\partial_{\pi} : K^{M}_{\bullet}{\bf F} \to 
(K^{M}_{\bullet}{\bf k}(\nu))(\Pi).
$$
Next, one defines maps
$$
\partial_{\pi}^0,\ \partial_{\nu} : K^{M}_{\bullet}{\bf F} 
\to K^{M}_{\bullet}{\bf k}(\nu),
$$
by
$$
\partial_{\pi}(x) = \partial^0_{\pi}(x) + 
\partial_{\nu}(x)\Pi,\leqno{(4.4)}
$$
which can be shown to be independent of the choice of $\pi$ such that $\nu
(\pi) = 1$.  This hinges on 
an explicit description of $\partial_{\nu}$ (see \cite{BT} Proposition 4.5). 
The (generalized) Tame symbol is
the map $\partial_{\nu} : K^{M}_m{\bf F} \to K^{M}_{m-1}{\bf
k}(\nu)$.  We also have that multiplication is graded
and skew commutative, 
and that $\deg \ell(a) = 1 = \deg \Pi$ for $a \in {\bf
F}^{\times}$.  For example, if we let $a = a_0\pi^i$, $b = b_0\pi^j$, so
that $\nu(a) = i,\ \nu(b) = j$, and let $\overline{a}_0$, $\overline{b}_0$
be the corresponding values in ${\bf k}(\nu)$, then
$$
\partial_{\pi}(a) = d_{\pi}(a) = \ell(\overline{a}_0) + i\Pi,\quad
\partial_{\pi}(b) = d_{\pi}(b) = \ell(\overline{b}_0) + j\Pi,
$$
so that $\partial_{\nu}(a) = i$ and $\partial_{\nu}(b) = j$.  Next,
as shown in \cite{BT},
$$
\partial_{\pi}(\{a,b\}) = (\ell(\overline{a}_0)+i\Pi)
(\ell(\overline{b}_0) +j\Pi)
$$
$$
= \big(\ell(\overline{a}_0)
\ell(\overline{b}_0)\big)\ + \ \big(\ell(\overline{a}_0^j)
-\ell(\overline{b}_0^i) + ij\ell(-1)\big)\Pi
$$
$$
\partial_{\nu}(\{a,b\}) =\ell\biggl((-1)^{ij}\frac{\overline{a}_0^j}
{\overline{b}_0^i}\biggr) = \ell\biggl(\overline{(-1)^{ij}
\frac{a^j}{b^i}}\biggr) = \ell\big(T\{a,b\}\big),
$$
where $T$ is the usual Tame symbol.  
Similarly, for a general product, we
have
$$
\partial_{\pi}(\{a^{(1)},\ldots,a^{(N)}\}) 
= (\ell(\overline{a}^{(1)}_0)+
k_1\Pi)\cdots (\ell(\overline{a}^{(N)}_0) + k_j\Pi)
$$
$$
\hskip.5in 
 = \big(\ell(\overline{a}^{(1)}_0)+ k_1\Pi\big)
\big(\partial_{\pi}^0(\{a^{(2)},
\ldots,a^{(N)}\}) + \partial_{\nu}(\{a^{(2)},
\ldots,a^{(N)}\})\Pi\big)
$$
$$
\hskip.1in 
= \biggl(\ell(\overline{a}^{(1)}_0)
\cdots\ell(\overline{a}^{(N)}_0)\biggr)
+ \biggl((-1)^{N-1}k_1\big(\ell(\overline{a}^{(2)}_0)\cdots
\ell(\overline{a}^{(N)}_0)\big) 
$$
$$
\hskip.5in + \big(\ell(\overline{a}^{(1)}_0) + k_1\ell(-1)\big)
\partial_{\nu}(\{a^{(2)},\ldots,a^{(N)}\})\biggr)\Pi,
$$
whence
$$
\ell(T^{(N)}\{a^{(1)},\ldots,a^{(N)}\}) :=
\partial_{\nu}(\{a^{(1)},\ldots,a^{(N)}\})
$$
$$
= \ell\biggl(\biggl\{\overline{a}^{(2)}_0,\ldots,
\overline{a}^{(N)}_0\biggr\}^{(-1)^{N-1}k_1}\biggl\{
(-1)^{k_1}\overline{a}^{(1)}_0,
T^{(N-1)}(\{a^{(2)},\ldots,a^{(N)}\})\biggr\}\biggr),
$$
where again $T^{(2)} = T$ is the usual Tame symbol above.

For our purposes we define the Tame symbol on $\underline{K}^{M}_{N,Z}$, 
by defining it's value at the generic point
of an irreducible codimension one  $D \subset Z$, viz.,
$$
T_{D}^{(N)}(\{\sigma^{(1)},\ldots,\sigma^{(N)}\}) :=
T^{(N)}(\{\sigma^{(1)},\ldots,\sigma^{(N)}\})\biggl|_{D} =
$$
$$ 
\biggl\{\overline{\sigma}_{0}^{(2)},\ldots,
\overline{\sigma}_{0}^{(N)}\biggr\}^{(-1)^{N-1}\nu_D(\sigma^{(1)})}
\ \times\hskip1in
$$
$$
\hskip.5in \biggl\{(-1)^{\nu_D(\sigma^{(1)})}\overline{\sigma}_{0}^{(1)}, 
T^{(N-1)}(\{\sigma^{(2)},\ldots,\sigma^{(N)}\})\biggr\},\leqno{(4.5)}
$$
where $\sigma^{(j)} = 
\pi_{D}^{\nu_D(\sigma^{(j)})}
\sigma_{0}^{(j)}$, $\pi_{D}$ a local equation 
of $D$ in $Z$, and $\overline{\sigma}_{0}^{(j)}$
the value of ${\sigma}_{0}^{(j)}$ in ${\CC}(D)^{\times}$.
Here it is important to understand that if $L_{j}/Z$
is the flat line bundle associated to $\sigma^{(j)}$, then
$\overline{\sigma}^{(j)}_{0}$ is a section of
$L_{j}\big|_{D}$. In other words, this calculation
occurs over the generic point of $D$, where one fixes
the choice of local equation of $D$ (but see Proposition 4.1(i)
below). Now note that $T^{(2)}(\{\sigma_{g},\sigma_{h}\}) = 
T(\{\sigma_{g},\sigma_{h}\})$ 
involves a section of a flat line bundle.  From this
it follows that $T^{(3)}$ can be defined with symbols of
sections of flat bundles. By induction, it follows that if $T^{(j+1)}$ 
involves symbols of sections of flat bundles, since this is the case
for $T^{(j)}$. If we work modulo $2$-torsion,  the formula
for $T_{D}$ in (6.5) simplifies somewhat, viz.,
$$
T_{D}^{(N)}(\{\sigma^{(1)},\ldots,\sigma^{(N)}\})
$$
$$
\equiv 
\biggl\{\overline{\sigma}_{0}^{(2)},\ldots,
\overline{\sigma}_{0}^{(N)}\biggr\}^{(-1)^{N-1}\nu_D(\sigma^{(1)})}
\biggl\{\overline{\sigma}_{0}^{(1)}, 
T^{(N-1)}(\{\sigma^{(2)},\ldots,\sigma^{(N)}\})\biggr\}
$$
$$
\equiv 
\prod_{j=1}^{N}\biggl\{\overline{\sigma}^{(1)}_{0},\ldots,\widehat{
\overline{\sigma}_{0}^{(j)}},\ldots,\overline{\sigma}_{0}^{(N)}
\biggr\}^{(-1)^{N-j}\nu_{D}(\sigma^{(j)})},
$$
where the latter equivalence modulo $2$-torsion follows by
induction. Since the real regulator is blind to torsion, it
makes sense to redefine the Tame symbol by the formula:
$$
T_{D}^{(N)}(\{\sigma^{(1)},\ldots,\sigma^{(N)}\}) :=
\prod_{j=1}^{N}\biggl\{\overline{\sigma}^{(1)}_{0},\ldots,\widehat{
\overline{\sigma}_{0}^{(j)}},\ldots,\overline{\sigma}_{0}^{(N)}
\biggr\}^{(-1)^{N-j}\nu_{D}(\sigma^{(j)})},\leqno{(4.5.1)}
$$
and accordingly redefine
$$
\underline{K}_{\bullet,Z}^{M} = \biggl\{\bigoplus_{j=0}^{\infty}
\underline{K}^{M}_{j,Z}\biggr\}\biggr/\left\{
\begin{matrix}2-{\rm torsion}\\
{\rm subgroup}\end{matrix}\right\}.\leqno{(4.5.2)}
$$
Thus $\underline{K}^{M}_{j,Z}$ will now be interpreted as
the corresponding group modulo $2$-torsion.
Note that $T^{(1)}$ is still the divisor map, $T^{(1)}$ and
$\underline{T}$ in Proposition 3.1 both agree 
on $\underline{K}^{M}_{2,Z}$ (as we are
working modulo $2$-torsion), and that $T^{(1)}\circ T^{(2)} = 0$. 
Quite generally, we prove the following:

\bigskip
\noindent
\begin{prop}
	{\it Assume given our
modified definition of $T^{(N)}$ in (4.5.1) above.

\bigskip
\noindent
{\rm (i)} The definition of 
$T^{(N)}$ does not depend on the local
equations defining the codimension one $D$'s in $Z$. 

\bigskip
\noindent
{\rm (ii)} Up to $2$-torsion,
$T^{(N)}(R^{N}) \subset R^{N -1}$
for all $N$.

\bigskip
\noindent
{\rm (iii)} $T^{(N)}\circ T^{(N+1)} = 0 \in \bigoplus
\underline{K}^{M}_{N-1,Z}$ 
for all $N$.}

\bigskip
\noindent
\end{prop} 

{\it Proof of Proposition 4.1.}\ The proof of Proposition 4.1
is a straightforward series of calculations. First of all, (i)
is true for the same reasons as in the standard case in (4.4) above.
If we let $\equiv$ have the meaning ``modulo $R^{N-1}$ and
$2$-torsion'', then the proof of (ii) follows from the calculations:
$$
T_{D}\{\sigma,-\sigma,\sigma^{(3)},\ldots,\sigma^{(N)}\}
$$
$$
\equiv \prod_{j=3}^{N}\{\overline{\sigma}_{0},-\overline{\sigma}_{0},\ldots,
\widehat{\overline{\sigma}^{(j)}_{0}},\ldots,\overline{\sigma}^{(N)}_{0}
\}^{\nu_{D}(\sigma^{(j)}(-1)^{N-j}}
\equiv 0.
$$
Likewise,
$$
T_{D}\{f,1-f,\sigma^{(3)},\ldots,\sigma^{(N)}\} \equiv 0.
$$
Here we use the fact that
for $D\subset |(f)|$, $f$ itself is a local equation.
Hence 
$$
(\overline{f}_{0},\overline{[1-f]}_{0}) = (1, [\pm]1)
:= \begin{cases} (1,1)&\text{if $\nu_{D}(f)\geq 0$}\\
&\\
(1,-1)&\text{if $\nu_{D}(f) < 0$}\end{cases}.
$$
To prove (iii), first observe that if $\sigma^{(1)},\ 
\sigma^{(2)}$ are rational sections of a bundle $L/Z$,  then
$$
{\div}\circ T\{\sigma^{(1)},\sigma^{(2)}\} = 0.
$$
This we proved earlier. This translates to saying
that 
$$
\sum_{E\subset D\ (\subset Z)}\big[\nu_{D}(\sigma^{(2)})\nu_{E}(
\overline{\sigma}^{(1)}_{0}) - \nu_{D}(\sigma^{(1)})
\nu_{E}(\overline{\sigma}^{(2)}_{0})\big]\{E\} = 0.\leqno{(4.7)}
$$

We now consider $E\subset D\subset Z$ and compute:
$$
T_{E}\circ T_{D}\{\sigma^{(1)},\ldots,\sigma^{(N)}\}
= \prod_{j=1}^{N}T_{E}\biggl\{\overline{\sigma}^{(1)}_{0},\ldots,
\widehat{\overline{\sigma}^{(j)}_{0}},\ldots,\overline{\sigma}^{(N)}_{0}
\biggr\}^{(-1)^{N-j}\nu_{D}(\sigma^{(j)})}
$$
$$
= 
\prod_{j=1}^{N}\biggl(\prod_{i<j}\biggl\{
\overline{\overline{\sigma}}_{00}^{(1)},\ldots,
\widehat{\overline{\overline{\sigma}}_{00}^{(i)}},\ldots,
\widehat{\overline{\overline{\sigma}}_{00}^{(j)}},\ldots,
\overline{\overline{\sigma}}_{00}^{(N)}\biggr\}^{(-1)^{i+j}
\nu_{D}(\sigma^{(j)})\nu_{E}(\overline{\sigma}^{(i)}_{0})}
$$
$$
\times\
\prod_{i>j}\biggl\{
\overline{\overline{\sigma}}_{00}^{(1)},\ldots,
\widehat{\overline{\overline{\sigma}}_{00}^{(j)}},\ldots,
\widehat{\overline{\overline{\sigma}}_{00}^{(i)}},\ldots,
\overline{\overline{\sigma}}_{00}^{(N)}\biggr\}^{(-1)^{i+j+1}
\nu_{D}(\sigma^{(j)})\nu_{E}(\overline{\sigma}^{(i)}_{0})}
\biggr)
$$
$$
=\prod_{j=1}^{N}\biggl(\prod_{i<j}\biggl\{
\overline{\overline{\sigma}}_{00}^{(1)},\ldots,
\widehat{\overline{\overline{\sigma}}_{00}^{(i)}},
\hskip2.5in
$$
$$
\hskip1in \ldots,
\widehat{\overline{\overline{\sigma}}_{00}^{(j)}},\ldots,
\overline{\overline{\sigma}}_{00}^{(N)}\biggr\}^{(-1)^{i+j}
[\nu_{D}(\sigma^{(j)})\nu_{E}(\overline{\sigma}^{(i)}_{0})
- \nu_{D}(\sigma^{(i)})\nu_{E}(\overline{\sigma}^{(j)}_{0})}\biggr).
$$
By summing over $E\subset D\subset Z$ and using (4.7), this
proves the proposition.
\qed

\bigskip
It follows then that one has a twisted Milnor 
complex of  the form 
$$
\underline{K}^{M}_{k-\bullet,X}: \quad \underline{K}_{k,X}^M \to
\bigoplus_{\cd_XZ = 1}\underline{K}_{k-1,Z}^M\to
$$
$$
\cdots\to
\bigoplus_{\cd_XZ = k-j-1}\underline{K}_{j+1,Z}^M\to
\bigoplus_{\cd_XZ = k-j}\underline{K}_{j,Z}^M\to\cdots
$$
$$
\cdots \to\bigoplus_{\cd_XZ = k-2}\underline{K}_{2,Z}^M\to
\bigoplus_{\cd_XZ = k-1}\underline{K}_{1,Z}^M\to
\bigoplus_{\cd_XZ = k}\underline{K}_{0,Z}^M\to 0.
\leqno{(4.7.1)}
$$
 Let
$$
H^{k-m}(\underline{K}^{M}_{k-\bullet,X}) :=\hskip2in
$$
$$
\frac{\ker T^{(m)}: \bigoplus_{\cd_XZ = k-m}
\underline{K}_{m,Z}^M \to \bigoplus_{\cd_XZ = k-m+1}
\underline{K}_{m-1,Z}^M}{T^{(m+1)}\big(\bigoplus_{\cd_XZ = k-2}
\underline{K}_{m+1,Z}^M\big)}.
$$

We now state our main result:

\bigskip
\noindent
\begin{thm}

\it Assume $m\geq 1$.  The current defined by
$$
\begin{matrix}\prod_{1}^{m}\sigma_j\in 
\prod_{1}^{m}{\Rat}^{\ast}(L_{j}/Z)\\
\cd_XZ = k-m\end{matrix}
\ \mapsto \biggl[\omega \mapsto 
$$
$$
\sum_{\ell=1}^m\int_Z(-1)^{\ell-1}
\log ||\sigma_\ell||
(d\log ||\sigma_1||\wedge\cdots \wedge \widehat{d\log ||\sigma_{\ell}||}
\wedge\cdots \wedge d\log ||\sigma_m||)\wedge \omega\biggr]
$$
descends to a cohomological map $\underline{r}_{k,m}$:
$$
\underline{CH}^{k}(X,m) := H^{k-m}(\underline{K}^{M}_{k-\bullet,X})\to  
\frac{\big\{H^{k-1,k-m}(X)\oplus H^{k-m,k-1}(X)\big\}
\cap H^{2k-m-1}(X,{\RR})}{D(k,m)},
$$
where 
$$
D(k,m) = \begin{cases} H^{2k-2}_{\rm alg}(X,{\QQ})
\otimes {\RR}&\text{if $m=1$}\\
0&\text{if $m>1$}\end{cases}.
$$
That is, this map does not depend on the choice of flat metrics on the 
respective flat bundles. Moreover, assume given $D$, smooth of
dimension $k-1$, and a morphism $f : D\to X$
such that $f(D)$ is in ``general position''. If $\omega
\in H^{m-1,0}(D)\oplus H^{0,m-1}(D)$ and $\xi\in  
H^{k-m}(\underline{K}^{M}_{k-\bullet,X})$ 
are given, then $\underline{r}_{k,m}(\xi)(f_{\ast}\omega)$
is induced by
$$
\begin{matrix}\prod_{1}^{m}\sigma_j\in 
\prod_{1}^{m}{\Rat}^{\ast}(L_{j}/Z)\\
\cd_XZ = k-m\end{matrix}
$$
$$
 \mapsto \ \biggl[\omega \mapsto \sum_{\ell=1}^m\int_{Z\cap D := 
 f^{-1}(Z)}(-1)^{\ell-1}
\log ||\sigma_\ell||
(d\log ||\sigma_1||\wedge
$$
$$
\cdots \wedge \widehat{d\log ||\sigma_{\ell}||}
\wedge\cdots \wedge d\log ||\sigma_m||)\wedge \omega\biggr].
$$
Furthermore, as a current acting on $E_{X}^{2n-2k+m+1}$, if
$\underline{r}_{k,m}(\xi)$ is $d$-closed, then
it restricts to the zero class in $H^{2k-m-1}(X\bs V)$ where
$V = \bigcup_{\alpha}Z_{\alpha}$ is 
the support of $\xi = \sum_{\alpha}\biggl(\prod_{1}^{m}
\sigma_{j,\alpha},Z_{\alpha}\biggr) \in 
H^{k-m}(\underline{K}^{M}_{k-\bullet,X})$.
In this case, $\underline{r}_{k,m}(\xi)$ lies in the
Hodge projected image $N^{k-m}H^{2k-m-1}(X,{\RR}) \to 
H^{k-1,k-m}(X) \oplus H^{k-m,k-1}(X)$.
\end{thm}

\bigskip
\noindent
{\it Proof.}\  While it would be redundant to give a complete proof 
of this theorem, it is important to explain the new ingredients
required to make the proof in the nontwisted case ([Lew1]) adaptible
to the twisted situation. Firstly, the theorem is already proven in 
the case $m=1$, this being the import of previous sections. Next,
if $||\sigma_{\ell}||\in {\RR}^{\times}$ is
constant for some $\ell$, then some standard estimates together
with a Stokes' theorem argument implies
that the corresponding regulator value
on closed forms $\omega$ is zero for $m\geq 2$. This leads to
independence of the flat metric, after quotienting out by $D(k,m)$,
for $m\geq 1$. To show for example that
$\big\{(\sigma,-\sigma,\sigma_{3},\ldots,\sigma_{m}),
Z\big\}$ or say $\big\{(f,1-f,\sigma_{3},\ldots,\sigma_{m}),Z\big\}$
goes to zero under the real regulator, amounts to reducing
to the case where the $\sigma_{j}$'s are rational functions
via some finite covering $Z^{\prime}\to Z$, similar to what we did
earlier in 3.2, and then applying the same arguments given
in \cite{Lew1}. That the current given in 4.2 is 
$\partial\overline{\partial}$-closed now follows from
the same proof as given in [Lew1]. Finally, the latter statement
of the theorem is rather 
easy to prove. Being $d$-closed implies that we have a cohomology
class on $X$, which restricts to a cohomology class
on $X\backslash V$; moreover the current clearly vanishes on those forms
compactly supported on $X\backslash V$. \qed

\section{Some examples}
In contrast to the various 
vanishing results in the literature for the
regulator in the nontwisted case
(\cite{MS1}, \cite{Ke1}-\cite{Ke2}, \cite{Co}, etc.), 
we exhibit some nonvanishing 
regulator results in the twisted case.

\bigskip
\noindent
{\it Regulator  on $H^{0}(\underline{K}^{M}_{2-\bullet,X})$ for a 
curve $X$.}\
Let $X$ be a compact Riemann surface of genus $g\geq 1$, and let $f, g\in 
{\CC}(X)^{\times}$ be given. Write $T\{f,g\} = 
\sum_{j=1}^{M}(c_{j},p_{j})$, where $p_{j}\in X$ and $c_{j}\in 
{\CC}^{\times}$. Then  $\prod_{j=1}^{M}c_{j} = 1$ by
Weil reciprocity.
Now fix $p\in X$, and let $L_{j}$ be a choice of line bundle 
corresponding to the zero cycle $p_{j}-p$. Since $\deg(p_{j}-p)
= 0$, $L_{j}$ is a flat line bundle. There exists rational
sections $\{\sigma_{j}\}$ of the flat bundles $\{L_{j}\}$ over $X$ 
such that $\div(\sigma_{j}) = p_{j}-p$. Thus one
can easily check by a Tame symbol calculation that 
$\xi := \{f,g\}\prod_{j}\{\sigma_{j},c_{j}\} 
\in H^{0}(\underline{K}^{M}_{2-\bullet,X})$. Note
that 
$$
d\big(\log|f|\log|g|\big) = \log|f|d\log|g| + \log|g|d\log|f|,
$$
and by a Stokes' theorem  argument together with standard estimates,
$$
\int_{X}\log|f|d\log|g|\wedge\w = -\int_{X}\log|g|d\log|f|
\wedge\w,\leqno{(5.1)}
$$
for any $d$-closed form 
$\omega\in  E_{X,{\RR}}^{1}$. Thus if either $f$ or $g$ were
constant, then both integrals in (5.1) would vanish. Applying
the same reasoning to the terms $\prod_{j}\{\sigma_{j},c_{j}\}$,
it easily follows that
$$
\underline{r}_{2,2}(\xi)(\omega) = \int_X\biggl[\log|f|d\log|g| - 
\log|g|d\log|f|\biggr]\wedge\omega,\leqno{(5.2)}
$$
namely the contribution of the terms $\prod_{j}\{\sigma_{j},c_{j}\}$
to the regulator current is zero.
One can always find $f$ and $g$ (see [Lew1]) such that the computation
in (5.2) is nonzero for general $X$. For a simple example,
consider this. Let $E$ be a general elliptic curve, $D$ another
general curve, and $X\subset E\times D$ a general hyperplane section.
Since $X$ dominates $E$, and that the real regulator for $E$
is nontrivial ([Blo3]), such an $f$ and $g$ can be found for
$X$ via pullback of corresponding rational functions
on $E$, and then by continuity, the same story will hold as
$X$ varies with general moduli. Thus in summary, one can find
general curves of genus $g >> 1$ for which the regulator
$$
\underline{r}_{2,2} : H^{0}(\underline{K}^{M}_{2-\bullet,X})
\to H^{1}(X,{\RR}),
$$ 
is nontrivial. This is in complete constrast to the situation
of the regulator in (0.1), viz.,  $r_{2,2} : \CH^{2}(X,2) \simeq
H^{0}_{\rm Zar}(X,{\cal K}_{2,X}) \to H^{1}(X,{\RR})$,
where it is known ([Co]) that $r_{2,2}$ is trivial for
sufficiently general $X$ of genus $g > 1$. Indeed, one can
naively carry out the same construction above to
arrive at a class $\xi \in H^{0}_{\rm Zar}(X,{\cal K}_{2,X})$
arising from rational functions $f,\ g$ on $X$ with
$T\{f,g\} = \sum_{j=1}^{M}(c_{j},p_{j})$, 
where $p_{j}\in X$ and $c_{j}\in {\CC}^{\times}$ and
$$
\int_X\biggl[\log|f|d\log|g| - 
\log|g|d\log|f|\biggr]\wedge\omega \ne 0.
$$
The issue boils down to finding a {\it rational}\
functions $h_{j}$ on $X$ for which div$(h_{j}) = N(p_{j}-p)$,
for some integer $N\ne 0$,
which is not in general possible. In fact the difficulty
of finding such $h_{j}$  amounts to finding 
torsion points on $X$, and is related to the known affirmative
answer to the Mumford-Manin conjecture.
For a discussion of the
relation of the Mumford-Manin conjecture to this
regulator calculation, the reader can consult
\cite{Lew1}, as well as the references cited there.

\bigskip
\noindent
{\it 5.3. Regulator  on $H^{1}(\underline{K}^{M}_{1-\bullet,X})$ for
a surface $X$.}\
For a simple example of a nontrivial twisted regulator calculation,
where the usual regulator vanishes, consider the case $X = M\times N$,
where $M$ and $N$ are smooth curves. If $M$ and $N$ are sufficiently 
general with $g(M)g(N) \geq 2$, then the image of the  regulator
in (0.1) vanishes (modulo the  group of algebraic cocycles) ([C-L1]). Now
suppose we are given a curve $C\subset M\times N$, $f\in  
{\CC}(C)^{\times}$, and $\omega\in H^{1}(M,{\CC})\otimes H^{1}
(N,{\CC}) \bigcap H^{1,1}(X,{\RR}(1))$ for which
$$
\int_{C}\omega\log|f| \ne 0.\leqno{(5.4)}
$$
Such a situation in (5.4) is fairly easy to arrive at
for general $M$ and $N$, by a deformation from a special
case situation. (For example, one can use a construction 
in \S7 of [Lew2], or $M\times N$ can be a general deformation
of a product of $2$ curves dominating a general Abelian surface,
together with the main results of [C-L2].)
Write
$$
{\rm div}(f)_{C} = \sum_{j=1}^{m}[(p_{j},q_{j})  - (s_{j},t_{j})].
$$
We can write
$$
(p_{j},q_{j})  - (s_{j},t_{j}) = [(p_{j},q_{j}) - (p_{j},t_{j})] 
+ [(p_{j},t_{j}) - (s_{j},t_{j})].
$$
Now put $D_{j} = \{p_{j}\}\times N$ and $K_{j} = M\times \{t_{j}\}$,
and observe that the degree zero divisor $(p_{j},t_{j}) - (p_{j},q_{j})$
on $D_{j}$ corresponds to a flat line bundle on $D_{j}$, and
likewise the degree zero divisor $(s_{j},t_{j}) - (p_{j},t_{j})$
on $K_{j}$ corresponds to a flat line bundle on $K_{j}$.
Consider Cartier divisors $\sigma_{j}$ on $D_{j}$,
$\eta_{j}$ on $K_{j}$, i.e. rational sections of the
respective flat line bundles, with
$$
{\div}(\sigma_{j}) = (p_{j},t_{j}) - (p_{j},q_{j}),
$$
$$
{\div}(\eta_{j}) = (s_{j},t_{j}) - (p_{j},t_{j}).
$$
Then
$$
\xi := (f,C) + \sum_{j=1}^{m}(\sigma_{j},D_{j}) + 
\sum_{j=1}^{m}(\eta_{j},K_{j})
\in H^{1}(\underline{K}^{M}_{1-\bullet,X}).
$$
[Note: As in the previous example, observe that one
cannot replace the  $\sigma_{j}$'s (resp. $\eta_{j}$'s)
by rational functions on the $D_{j}$'s (resp. $K_{j}$'s),
even if one replaces $(p_{j},t_{j}) - (p_{j},q_{j})$ and
$(s_{j},t_{j}) - (p_{j},t_{j})$ by nonzero integral multiples.]
Moreover the pullback of $\omega$ to $D_{j}$ and $K_{j}$ is zero.
Thus:
$$
\underline{r}(\xi)(\omega) = \int_{C}\omega\log|f| \ne 0.
$$
Finally, observe that for general $M$ and $N$,
$\big\{H^{1}(M,{\QQ})\otimes H^{1}(N,{\QQ})\big\} \cap
H^{2}_{\alg}(M\times N,{\QQ}) = 0$. Thus
while  the image of the  regulator
in (0.1) for $M\times N$ vanishes (modulo the  group of 
algebraic cocycles), the twisted regulator does not vanish.

Next, we will introduce twisted Hodge $\Dd$ conjecture. First we will introduce the strong twisted Hodge $\Dd$ conjecture, and twisted Hodge $\Dd$ conjecture, and introduce the philosophy that we shall follow in proving it for certain family of smooth projective varieties. 

\begin{conj}[\textbf{Strong Twisted Hodge $\Dd$ conjecture}]
The morphism 
$$\underline{r}_{k,m} \otimes \mathbb{R} : \underline{CH}^{k}(X,m) \otimes \mathbb{R} \rightarrow \frac{\big\{H^{k-1,k-m}(X)\oplus H^{k-m,k-1}(X)\big\}
\cap H^{2k-m-1}(X,{\RR})}{D(k,m)}$$

is surjective.

\end{conj}

\begin{conj}[\textbf{Twisted Hodge $\Dd$ conjecture}]
The morphism 
$$\underline{r}_{k,1} \otimes \mathbb{R} : \underline{CH}^{k}(X,1) \otimes \mathbb{R} \rightarrow \frac{\big\{H^{k-1,k-1}(X)\oplus H^{k-1,k-1}(X)\big\}
\cap H^{2k-2}(X,{\RR})}{D(k,1)}$$

is surjective.

\end{conj}

Next, we will explain the philosophy of proving twisted Hodge $\Dd$ conjecture. First let us introduce the following group. 

\begin{defn}
A twisted Chow group $\zeta \in CH^{k,m} \otimes \mathbb{R}$ is said to be indecomposable if $\underline{r}_{k,m}(\zeta) \neq 0$. 	
\end{defn}

\begin{defn}
We shall define the group $\ul{z}^k(X,1)_P$, of precycles of twisted group defined as follows:

\begin{align*}
  \underline{z}^k(X,1)_P := 
  \left\{
    \mathlarger{\sum}_i (\sigma_i, \|\ \|_{L_i}, V_i) : 
    \begin{aligned}
      \operatorname{codim}_X V_i = k -1,\\
      (L_i/V_i, \|\ \|_{L_i}) \text{ \rm{flat}}, \\
      \sigma \in \operatorname{Rat}^*(L_i), 
    \end{aligned} \ \ \ 
  \right\}
\end{align*}	
\end{defn}

We are interested in the following questions: first, whether we can create, in a systematic way, an indecomposable twisted cycle. Secondly, we are interested in the conjectures stated above.

\begin{rem}
One of the philosophies that we shall adopt in attacking this problem is the following, first approach is that we start with a precycle $\zeta \in z^{k}(X,1)_P$, and we wish to add $\gamma \in \underline{z}^{k}(X,1)_P$ such that $\omega = \zeta + \gamma \in \underline{z}^k(X,1)$, and moreover, $\underline{r}_{k,1}(\omega) \neq 0$. This philosophy will play a critical role in our results that we will start with in the next section. 
\end{rem}

\end{section}

\begin{section}{General smooth projective surfaces}
In this section we will explore results we proved for smooth projective surfaces. We will prove that given a general smooth projective surface, the twisted Hodge-$\Dd$ holds.

\begin{thm}
	Suppose we have a smooth and proper family $\rho : \mathcal{X} 
	\rightarrow \mathcal{S}$ together with $X_t = \rho^{-1}(t)$ with $0 \in \mathcal{S}$ fixed.Consider the real regulator
	$r_{2,1} : \CH^2(X_0,1) \otimes_{\mathbb{Z}} \mathbb{R} \rightarrow H^{1,1}(X_0,\mathbb{R})$.\ Set \ $\lambda_0 = \dim_{\mathbb{R}}(r_{2,1}(\CH^2(X_0,1) \otimes_{\mathbb{Z}} \mathbb{R}))$.Then on a non-empty real analytic Zariski open subset $U$ $\subset S$, $t \in U$ implies $\dim_{\mathbb{R}}(\ul{r_{2,1}}(\ul{\CH}^2(X_t,1) \otimes \mathbb{R})) \geq \lambda_0$.
\end{thm}
\end{section}

$\newline$
\begin{proof}

Choose $\{\zeta_1,\ldots, \zeta_{\lambda_0}\} \in {\CH}^2(X_0,1)$ and consider their regulator images $\{r(\zeta_1), \ldots,r(\zeta_{\lambda_0})\}$. Consider any $\zeta_M$, denote it by $\zeta$. Then it follows that 
$\zeta = \Sigma_j^{N_0} (f_j,D_j)$, where $\Sigma_j \Div_{D_j}(f_j) = 0$. We look at the following dictionary $\zeta = \Sigma_j^{N_0} (f_j,D_j)$ corresponds to $(f,D) = (f_j,D_j)$.

Let $H = \textrm{hyperplane section of } \mathcal{X} \subset \mathcal{\bar{X}} \subset \mathbb{P}^N$ for some $N$. Note that $\mathcal{\bar{X}}$ is smooth projective, as we can always apply a sequence of blowups.  Define $H_t = H \cap X_t$. Choose  $m$ big enough such that $mH_0 - D {\sim}_{\mathrm{rat}} D^{\prime}$, for $D^{\prime}$ effective. This follows from $m >> 1$, $mH_0 - D$ is very ample (see \cite{GH}). Hence there exists a regular section $u$ of $mH_0 - D$. It follows $mH_0 \sim_{\mathrm{rat}} D + D^{\prime}$. We may assume that $D^{\prime}$ has no components in common with $D$ by Chow's moving lemma. For $m >> 1$ we have that $mH_0$ is a complete linear system, it follows that one can find a hypersurface $V(h) \cap X_0$ such that $V(h) \cap X_0 = D + D^{\prime}$. $V(h) \cap X_0$ has a deformation in $\mathcal{S}$. Set $D_t = V(h) \cap X_t$, where $D_0 = D + D^{\prime}$. $D_0$ deforms in $\mathcal{S}$. We can lift $h$ to an element $f \in \mathbb{C}(\mathbb{P}^N)^{\times}$. Let $(g,D)$ be given such that $g \in \mathbb{C}(D)^{\times}$.

 We can choose rational function $k \in \mathbb{C}(X_0)^{\times}$ such that $k$ when restricted to $D$ is non-zero when restricted to each irreducible component of $D$, and $k$ restricted on $D^{\prime}$ is zero. We introduce the following $\epsilon$ construction: $$\biggl(h_{\epsilon} = 1 + \frac{(f - 1)k}{k + \epsilon}\biggl) \in  \mathbb{C}(\mathcal{X})^{\times}$$

Note that $\zeta$ under the dictionary corresponds to $\newline$$(h_{\epsilon},D_0 )= (h_{\epsilon}, D + D^{\prime})$. Note that $(h_{\epsilon},D + D^{\prime}) = (1, D^{\prime}) + (h_{\epsilon},D)$ and as $\epsilon \mapsto 0$ we have $(h_{\epsilon},D) \mapsto (f,D)$.

$\newline$Put $h_{\epsilon,t} = h_{\epsilon}|_{X_t}$ and consider a deformation, $(h_{\epsilon,t},D_t)$ with $D_{t_0} = D + D^{\prime}$. On $mH_t$ we have the following:
$$\Div(h_{\epsilon,t}^j) = \Sigma_{i = 1}^M (p_i(t)^j - q_i^j(t))$$

Remark:

$$\Sigma_{j = 1}^N \Div_{D_t}(h_{\epsilon,t}^j) \mapsto 0 \ \textrm{as } (t,\epsilon) \mapsto (0,0).$$

By intersecting $H = H_{\bar{\mathcal{X}}} $ with $X_t$ we can find curves $C_{i,t}^{j} \subset X_t$ such that

$$p_i^j(t) - q_i^{j}(t) \in C_{i,t}^{j} \textrm{ for } j = 1,\ldots, N \textrm{ such that}$$

$\newline$$C_{i,t}^j$ is smooth along $p_i^j(t)$ and $q_i^j(t)$.   This follows from Bertini's theorem together with dimension type argument on the dual projective space. Bertini's theorem tells us generic elements are smooth away from the base locus. On the base locus we have a dimension type argument. For completeness we include this dimension type argument. Consider a projective embedding $X \hookrightarrow \mathbb{P}^N$. We will assume minimality, i.e. we will assume $N$ is minimal such that $X \not \subset \mathbb{P}^{N - 1}$. Since $X$ is a surface it follows that $\mathbb{P}^{N - 1} \cap X$ is a curve. For any point $q \in X$, and $q \in \mathbb{P}^{N - 1}$, $\mathbb{P}^{N - 1} \cap X$ is singular at q $\iff$ $\overline{T_p(X)} \cong \mathbb{P}^2 \subset \mathbb{P}^{N - 1}$. Since $\overline{T_p(X)}$ is determined by $3$ distinct non-collinear points. Thus $\{t \in \mathbb{P}^{N, *} : \overline{T_p(X)} \subset \mathbb{P}_t^{N - 1}\}\cong \mathbb{P}^{N - 3}$. On the other hand, $\{ t \in \mathbb{P}^{N,*} : p,q \in \mathbb{P}_t^{N - 1} \}$ is of codimension 2 in $\mathbb{P}^{N, *}$. Thus by a dimension count, one can find $t \in \mathbb{P}^{N,*}$ such that $C := \mathbb{P}_t^{N - 1} \cap X$ is smooth at $p,q \in C$ as well as the full base locus of points. Bertini's theorem tells us that $C$ is connected and smooth outside the base locus. Since $C$ is already smooth at base locus, it is therefore smooth and connected, hence irreducible too. Construct a Cartier divisor $\beta_{t,\epsilon}^j$ on the curve $C_{t,\epsilon}^j$ such that 
$$\Div_{C_{t,\epsilon}^j}(\beta_{t,\epsilon}^j) = p_j(t,\epsilon) - q_j(t,\epsilon).$$

$\newline$The line bundles $L_j$ associated to $(\beta_{t,\epsilon}^j)$ are flat over $C_{t,\epsilon}^j$. Finally to each cycle $\zeta_1,\ldots, \zeta_{\lambda_{0}}$ we consider the real regulator $r(\zeta_1),\ldots, r(\zeta_{\lambda_{0}})$, and we deduced that we can complete those to twisted cycles. Finally if we check for linear independence, then we are done. But this follows as we can arrange for $\|\beta_{t,\epsilon}^j \|_L \mapsto 1$ as $(\epsilon,t)^j \mapsto (0,t_0) \in \mathbb{R}\times \mathbb{P}^{N,*}$,  as we have that the real regulator images of $\{r(\zeta_1), \ldots,r(\zeta_{\lambda_0})\}$ are linearly independent. A determinant calculation is a continuous function and being non-zero at a point implies it is non-zero in a neighbourhood of that point.
\end{proof}

\begin{cor}
	Twisted Hodge D conjecture holds for general smooth projective surfaces
\end{cor}

\begin{proof}
	We know the Hodge D conjecture holds for surfaces of maximum Picard rank. Since we showed that general smooth project surface, we have $\dim_{\mathbb{R}}(\ul{r_{2,1}}(\ul{\CH}^2(X_t,1) \otimes \mathbb{R})) \geq \lambda_0$, so surjectivity follows. 
\end{proof}

\begin{rem}
	We note the following difference between higher Chow cycles and twisted cycles: For sufficiently general curves $C_1$ and $C_2$, with $g(C_1)g(C_2) \geq 2$, the image of the real regulator is zero. However, according to Corollary 2.2, the image of the twisted regulator isn't zero. In the next couple of pages, we explicitly construct indecomposables for the product of curves, and we show surjectivity explicitly for the product of elliptic curves.
\end{rem}

\begin{lem}
Given $X = C_1 \times C_2$, $C_j$ smooth curve, $D \subset X$ irreducible curve $f \in \mathbb{C}(D)^{\times}$, then $(f,D)$ can be completed to twisted algebraic cycle $\zeta$ such that:

$$\ul{r}_{2,1}(\zeta)(\omega) = \int_{D} \log |f| \omega$$	
\end{lem}

\begin{proof}
We have $\Div(f) = \Sigma n_i(p_i,q_i)$ with $\Sigma n_i = 0$. Set $D_i = C_1 \times \{q_i\}$. Fix $p \in C_1$ such that $p_i - p \sim_{\rm{alg}} 0$ on $C_1$. It follows $p_i - p \in \CH^1_{\rm{hom}}(C_1)$. It follows that there exists a flat line bundle $L_i^{\otimes n_i}$ on $C_1$ and $\sigma_i \in \operatorname{Rat}^*(L_i)$ such that $\Div(\sigma_i) = (p_i - p)n_i$.

Consider $\Div_D(f) - \Sigma_i \Div_{D_i} (\sigma_i,D_i) = (p,\Sigma n_i q_i)$. For $\omega \in H^2_{\rm{tr}}(C_1 \times C_2)$ we have $\omega|_{D_i} = 0$ by Kunneth's formula. It follows that $\Sigma n_i q_i \sim_{\rm{alg}} 0$ on $C_2$. Set $D_0 = \{p\} \times C_2$. By the same reasoning as above we have flat line bundle $L$ on $C_2$ such that $\Div_{D} \sigma = (p, \Sigma n_i q_i)$. Therefore $(f,D) - \Sigma_i (\sigma_i,D_i) - (\sigma,D) \in \underline{\CH}^2(X,1)$. We have $\omega|D = 0$. Hence it follows that
$$ \int_D \omega \log|f| + \int_{D} \omega \log\|\sigma_0\| + \mathlarger{\sum_i} \int_{D_i} \omega \log\|\sigma_i\| = \int_D \omega \log|f|$$
\end{proof}

\begin{thm}
Let $X = E_1 \times E_2$, be product of smooth elliptic curves, given by the Weierstrass equations by first setting for $j= 1,2$ 
$$F_j = y_j^2 - x_j^3 + b_jx_j + c_j \ and \ X = V(\bar{F_1},\bar{F_2}) \simeq E_1 \times E_2$$
 then the Hodge D conjecture and twisted Hodge D conjecture holds, i.e. ${r}_{2,1} : {\CH}^2(X,1) \otimes_{\mathbb{Z}} \mathbb{R} \rightarrow H^{1,1}(X,\mathbb{R})$ and $\ul{r}_{2,1} : \ul{\CH}^2(X,1) \otimes_{\mathbb{Z}} \mathbb{R} \rightarrow H^{1,1}(X,\mathbb{R})$ are surjective.
\end{thm}
\begin{proof}

Let $D = \{x_1x_2+y_1y_2 = 0\}$ and consider $f_1 = x_1^2x_2+\I$ and $f_2 = x_1^2x_2 + 1$. Set $\omega_i = \frac{dx_i}{y_i}$,  $\eta_1 = \omega_1 \wedge \bar{\omega_2} + \bar{\omega_1} \wedge \omega_2$, and $\eta_2 = \I(\omega_1 \wedge \bar{\omega_2} - \bar{\omega_1}\wedge \omega_2)$. We will prove that $\int_D \eta_1 \log|f_1| \neq 0$ and $\int_D \eta_2 \log|f_2| \neq 0$. Also, we have $\int_D \eta_1 \log|f_2| = \int_D \eta_2\log|f_1| = 0$.

We will prove the above using two degeneration argument. It is clear that X varies with $t = (b_1,c_1,b_2,c_2)$. Set $D_t = X \cap V(x_1x_2 + y_1y_2)$. We will consider two deformations first deforming $t = (b_1,c_1,b_2,c_2)$ to $(b_1,0,b_2,0)$ then to $(0,0,0,0)$. On $t = (b_1,0,b_2,0)$ we have that the equations of $E_j$ becomes:

$$y_j^2 = x_j^3 + b_jx_j.$$

On $D_t$ we have we have the following equations $x_1^2x_2^2 = y_1^2y_2^2 = x_1x_2(x_1^2 + b_1)(x_2^2 + b_2)$. Therefore $D_t$ can be decomposed as 
$$D_t = E_1 \times [1,0,0] + [1,0,0]\times E_2 + D_t^{\prime}$$

Because on $D_t^{\prime}$ we have $x_1x_2 \neq 0$ by cancelling out $x_1x_2$ we have the following equations $x_1x_2 = (x_1^2 + b_1)(x_2^2 + b_2)$ on $D_t^{\prime}$. Moreover, on $E_1 \times [1,0,0]$ and $[1,0,0] \times E_2$ the integral equates to zero. Therefore we have 
$$ \int_{D_t} \eta_i \log|f_i| = \int_{D_t^{\prime}} \eta_i \log|f_i|$$

Now deforming $(b_1,0,b_2,0)$ to $(0,0,0,0)$ we have that the equations of $E_1$ and $E_2$ becomes
$$y_j^2 = x_j^3$$

Therefore we have the following equation:

$$(x_1x_2)^2 = (y_1y_2)^2 = y_1^2y_2^2 = x_1^3x_2^3$$

If $x_1x_2 = 0$ it follows that $\log|f_i| = 0$. Therefore we can assume that $x_1x_2 \neq 0$. It follows that from the equations 

$$(x_1x_2)^2 = (y_1y_2)^2 = y_1^2y_2^2 = x_1^3x_2^3$$

that $x_1x_2 = 1$. Therefore $y_1^2y_2^2 = 1 \implies (y_1y_2)^2 = 1$. It follows that $y_1y_2 = 1$ or $y_1y_2 = -1$. If $y_1y_2 = 1$ then it follows that $x_1x_2 + y_1y_2 = 0 \implies y_1y_2 = -1$ this gives us a contradiction. Therefore $y_1y_2 = -1$. Hence the equations we have the following decomposition of $D_t^{\prime}$

$$D_t^{\prime} = E_1 \times [1,0,0] + [1,0,0]\times E_2 + D_t^{\prime \prime}$$

Where $D_t^{\prime \prime} = X \cap V(x_1x_2 - 1)$. Recall we are working under the assumption that $x_2 = x_1^{-1} \implies dx_2 = -x_1^{-2}dx_1$. Recall

 $$\eta_1 = \frac{dx_1}{y_1} \wedge {\frac{d\overline{x}_2}{\overline{y}_2}}+ {\frac{d\overline{x}_1}{\overline{y}_1}}\wedge \frac{dx_2}{y_2}$$
 Plugging in $x_2 = x_1^{-1}$,$dx_2 = -x_1^{-2}dx_1$, and $y_2 = -y_1^{-1}$ into the calculation below with the elliptic curve equations $y_i = x_i^3$ we get the following computations:

 \begin{equation} \nonumber
\begin{split}
\frac{d \overline{x}_1}{\overline{y}_1} \wedge \frac{dx_2}{y_2} = \overline{
\frac{\overline{x}_1}{|x_1|^3} dx_1 \wedge d\overline{x}_1}  \\
	= \frac{x_1}{|x_1|^3}d\overline{x}_1 \wedge dx_1 \\
	= \frac{-x_1}{|x_1|^3}dx_1 \wedge d\overline{x}_1 
\end{split}
\end{equation}

$$\therefore \eta_1 = \frac{\overline{x}_1 - x_1}{|x_1|^3} dx_1 \wedge d \overline{x}_1 = \frac{-2{\rm{Im}}(x_1)}{|x_1|^3} dx_1 \wedge d \overline{x}_1$$

Now using the same computations for $\eta_2$ we get the following:

$$\therefore \eta_2 = \frac{(\overline{x}_1 + x_1)}{|x_1|^3} dx_1 \wedge d \overline{x}_1 = \frac{2{\rm{Re}}(x_1)}{|x_1|^3} dx_1 \wedge d \overline{x}_1$$

We first prove that $\int_D \eta_1 \log|f_1| \neq 0$. We have two degeneration argument. By the degeneration argument we are really integrating over $\mathbb{C}$. Therefore $\int_D \eta_1 \log |f_1| \rightarrow \int_{\mathbb{C}} \log|x_1 + \I| \frac{-2{\rm{Im}}(x_1)}{|x_1|^3} dx_1 \wedge  d \overline {x}_1$. We will not worry about the negative sign in front of the integral as it doesn't contribute to the integral being zero. The integral over the upper half-upper $\mathbb{C}$ is the same as the integral over the lower half plane. Therefore it suffices to look at the integral  $\int_{\mathbb{H}} \log|\frac{x_1 + \I}{\I + \overline{x}_1}| \frac{2{\rm{Im}}(x_1)}{|x_1|^3} dx_1 \wedge d\overline {x}_1$. Over the upper half plane $\frac{{\rm{Im}}(x_1)}{|x_1|^3} > 0$. We have $ \log|\frac{x_1 + \I}{\I + \overline{x}_1}| > 0 \iff |\frac{\I + x_1}{\I + \overline{x}_1}| > 1 \iff {\rm{Im}}(x_1) > 0$. Hence $\int_{\mathbb{H}} \log|\frac{x_1 + \I}{\I + \overline{x}_1}| \frac{2{\rm{Im}}(x_1)}{|x_1|^3} dx_1 \wedge d\overline {x}_1 > 0 \implies \int_{\mathbb{C}} \log|\frac{x_1 + \I}{\I + \overline{x}_1}| \frac{2{\rm{Im}}(x_1)}{|x_1|^3} dx_1 \wedge d\overline {x}_1 > 0.$ On the other hand if we look at $\int_{\mathbb{H}} \log|\frac{1 + x_1}{1 + \overline{x}_1}| \frac{2{\rm{Im}}(x_1)}{|x_1|^3} dx_1 \wedge d\overline {x}_1$, then $|\frac{1 + x_1}{1 + \overline{x}_1}| = |\frac{1 + x_1}{(\overline{1 + x_1})}| = 1$. Therefore the integral vanishes.  By a symmetric argument we have $\int_D \eta_2 \log|f_1| = 0$ and $\int_D \eta_2 \log|f_2| \neq 0$.

From lemma 4.2 it follows $(f_i,D)$ can be completed towards twisted cycle $\zeta_i$ such that $\ul{r}_{2,1}(\zeta_i)(\omega) = \int_D \omega \log |f_i|$. Since 
$$\det\begin{bmatrix}
    \int_D \eta_1 \log|f_1|       & \int_D \eta_1 \log|f_2| \\
    \int_D \eta_2 \log|f_1|       & \int_D \eta_2 \log|f_2|
\end{bmatrix} \neq 0$$

Hence it follows that that $\ul{r}_{2,1}(\zeta_1)$ and $\ul{r}_{2,1}(\zeta_2)$ provides a basis for $\mathbb{R}^2$. Since $H^{1,1}(X,\mathbb{R}) \simeq \mathbb{R}^2$ therefore $\ul{r}_{2,1}$ is surjective. 	
\end{proof}

\begin{cor}
The twisted Hodge $\Dd$ conjecture is true for Elliptic fourfolds.	
\end{cor}

\begin{proof}
Follows right away as we constructed a basis.	
\end{proof}

\begin{rem}
We can see that the most important aspect of the proof is that we know that on any subvarieties, all points are homologically equivalent. This is the most important aspect. This changes for higher twisted cycles, as we will see in the next section. We need cycles that behave well homologically. That will turn out to be a complete intersection.
\end{rem}

\begin{section}{Higher dimension twisted cycles}

	Below we will present the settings and the theorems that we will be using. After that we will explain the idea of the proof, and the main point of this paper. Let $X$ be a smooth projective fourfolds. We want to construct $\overline{\zeta} \in \ul{z}^k(X,1)_P$, that have non-zero regulator such that we can construct $\zeta \in \ul{z}^k(X,1)$ preserving the non-triviality of the regulator value. That is, we want to construct a twisted indecomposable on $X$. In this paper we will focus on $\ul{z}^3(X,1)$. 

In this paper and construction we will be dealing with three types of geometric objects, divisors in $X$, and surfaces in $X$, and curves in $X$.

	In the case of fourfolds above we are interested to complete precycles in $\overline{\zeta} \in \ul{z}^3(X,1)_P$. Let $S = H_1 \cap H_2 \cap X$, intersection of two smooth hyperplane sections $H_1 \cap X$, and $H_2 \cap X$.

We work in the following settings, let $X = E_1 \times E_2 \times E_3 \times E_4$, is a general product of four elliptic curves. First we define horizontal and vertical cycles. They are cycles in $X$, where it is a product of subvarieties where one of the products are points in subfactors of $X = E_1 \times E_2 \times E_3 \times E_4$, and the other subfactor is subvariety of subfactor of the projection of $X$, defined more precisely below. 

We define two independent projections as projection from $X = E_1 \times E_2 \times E_3 \times E_4$ to subfactors of $X$, where these projections go to different subfactors. For instance the projections $P_1 = Pr_{1,2,3} : X \twoheadrightarrow E_1 \times E_2 \times E_3$ and $P_2  = Pr_{4} : X \twoheadrightarrow E_4$ are two independent projections. Another example if we let $P_1 = Pr_{1,2} : X \rightarrow E_1 \times E_2$, and $P_2 = Pr_{3,4} : X \rightarrow E_3 \times E_4$ are two independent projections.  

\begin{defn}
Let $P_1 : X = E_1 \times E_2 \times E_3 \times E_4 \twoheadrightarrow P_1(X)$ and $P_2 : X \twoheadrightarrow P_{2}(X) $ be two indepedent projections.
Let $B$ be the subvariety of $X$, where $P_1(B) = p$, a point in a subfactor of $X$, and $P_2(B) = V$ is a subvariety of subfactor of $X$, such that $B = p \times V$. Define $PHVC = B = p \times V$, and HVC to be formals sums of $PHVC$.
\end{defn}

\begin{ex}
	Let $p_1$ be a point in $E_1$, then $p_1 \times E_2 \times E_3 \times E_4$ are HVC. Another example is if $V$ is a curve in $E_2 \times E_3 \times E_4$, then $p_1 \times V$ is a HVC. There are other HVC, based on the definition above as well, for instance if $V$ is a curve in $E_3 \times E_4$, and $p_1 \times p_2 \times V$, is HVC, where $p_1 \times p_2$ are points in $E_1 \times E_2$. Other examples are possible, so $HVC$ are things of the form $p \times V$, where $p$ are points not necessary just in one factor of $E_j$, and $V$ is a subvariety in a projection distinct from the first projection. 
\end{ex}

In the theorem below we state it with respect to arbitrary smooth projective variety $X$. 
\begin{thm}
	Let $X$ be a smooth projective variety, and $Y$ is a normal subvariety of X such that $dim(Y) \geq 2$, that is complete intersection of hypersurfaces of $\mathbb{P}^N$, intersected with $X$ such that $Y$ is not necessarily smooth, then $Pic^0(X) \cong Pic^0(Y)$.
\end{thm}

\begin{thm}
Let $X$ be a smooth projective variety of dimension at least 4, and $Y$ be a smooth, complete intersection intersection of dimension at least 3. The restriction maps of Picard groups $\Pic(X) \rightarrow \Pic(Y)$ is an isomorphism. 
\end{thm}

Now we go back to the settings, where $X = E_1 \times E_2 \times E_3 \times E_4$, product of  of general elliptic curves $E_j$, consider the following proposition.

\begin{prop}
Any divisor $A_X$ on $X$ is homologically equivalent to $HVC$. It follows that there exists triple $(\gamma,X,M)$ such that $\div(\gamma)_X = A_X - HVC$.
\end{prop}

\begin{proof}
By Kunneth decomposition, and period calculations, we can show that for any divisor $A$, we have 

\begin{align*}
	A \sim_{hom,X} & n_1 \cdot p_1 \times E_2 \times E_3 \times E_4 + n_2 \cdot E_1 \times p_2 \times E_3 \times E_4 \\
	& + n_3 \cdot E_1 \times E_2 \times p_3 \times E_4 + n_4 \cdot E_1 \times E_2 \times E_3 \times p_4 =: HVC
\end{align*}

where $p_j \in E_j$, and as $X$ is smooth, it follows there exists non-zero rational section $\gamma$ on flat line bundle $M$ such that $\div(\gamma)_X = A - HVC$. 
\end{proof}

Thus we have the following corollary, below we $C$ stands for cycle. 
\begin{cor}
Given any $A_{X,C} \in z^1(X)$, effective, then it follows there exists triple $(\gamma_C,X,M_C)$ such that $\div(\gamma)_X = A_{X,C} - HVC$.
\end{cor}

\begin{defn}
Let $S$ be any arbitrary surface in $X$. Define horizontal and vertical cycles in $S$ as the pull-back of $HVC$ on $X$ to $S$, i.e. 
\begin{align*}
	HVC := 	&\ n_1 \cdot p_1 \times E_2 \times E_3 \times E_4 \cap S + n_2 \cdot E_1 \times p_2 \times E_3 \times E_4 \cap S \\
	 \ &+ \ n_3 \cdot E_1 \times E_2 \times p_3 \times E_4 \cap S \ + \ n_4 \cdot E_1 \times E_2 \times E_3 \times p_4 \cap S
\end{align*}

\end{defn}

Now we go back to the setting where $S = H_1 \cap H_2 \cap X$.

\begin{cor}
Given a divisor $A \subset S$ a divisor, then there exists triple $(\sigma,S,L)$ such that

\begin{align}
 	\div(\sigma)_S = A - (n_1 p_1 \times W_1 + n_2 p_2 \times W_2 + n_3 p_3 \times W_3 + n_4 p_4 \times W_4) = A - HVC. 
 \end{align}
 
\end{cor}

\begin{proof}
	We have that $Pic^0(X) = Pic^0(S)$, by theorem $7.3$, and since any divisor on $X$ is homologically equivalent to cycle of the form $HVC$ defined above, and these cycles restrict to cycles of the form $HVC$ defined in definition 6.7.

It follows $HVC$ are precycles the curves that are of the form given in $(1)$. It follows there exists triple $(\sigma,S,L)$ such that $\div(\sigma)_{S} = A - HVC$.
\end{proof}

\end{section}

\begin{section}{Hypersurface arrangements}

Below we introduce hyperplane surface arrangements, which is a construction that allows us to construct twisted cycles from specific precyles. 

We will divide our construction into sequence of generalization. First given a precycle of the following form $A = H_3 \cap S\sim_{hom,S} HVC$, where $H_3 \cap X$ is a hyperplane section. Notice $A = H_1 \cap H_2 \cap H_3 \cap X$. This will induce the triple $(\sigma,S,L)$ such that $\div(\sigma)_S = A - HVC$. First we take $HVC$ homological representative of $A$. Notice the reason we pick such $A$ is because $A$ deforms easily. 

\begin{lem}
	We can construct a precycle $\overline{\zeta}_1$ from the surfaces $S$ such that
	
	\begin{enumerate}
		\item $\displaystyle \underline{r}_{3,1}(\overline{\zeta}_1)(\eta) = \int_{S} \log \norm*{\sigma} \eta$.
		\item $\div(\overline{\zeta}_1) = HVC$.
	\end{enumerate}

\end{lem}

To prove Lemma 7.1, we will need the following Proposition, 

\begin{prop}
There exists a triple $(\beta,Z,J)$ where the divisor satisfies $\div(\beta)_Z = A - HVC$ over $Z = Pr_{1,2,3}(A) \times E_4$.

\end{prop}

\begin{proof}

	By definition, $\div(\sigma)_S = H_3 \cap X \cap S - HVC = H_3 \cap H_2 \cap H_3 \cap X - HVC = A - HVC$. Recall that $S = H_1 \cap H_2 \cap X$, from which it follows that $A := H_3 \cap S = H_1 \cap H_2 \cap H_3 \cap X$, is a complete intersection of hyperplane sections. Any divisor $H_3 \cap X$ in $X$ is homologically equivalent to horizontal and vertical cycle. It follows that 
	\begin{align*}
		H_3 \cap X \sim_{hom,X} HVC =\  &n_1p_1 \times E_2 \times E_3 \times E_4 + n_2E_1 \times p_2 \times E_3 \times E_4 \\
		& n_3E_1 \times E_2 \times p_3 \times E_4 + n_4E_1 \times E_2 \times E_3 \times p_4.
	\end{align*}
	
We have inclusion map $Z = Pr_{1,2,3}(A) \times E_4 \rightarrow X$. We are interested in computing the pull-back of $H_3 \cap X$ and $HVC$ from $X$ to $Z$. If $Z$ isn't smooth, then we can pull-back with respect to desingularization $\tilde{Z} \rightarrow Z$, and then push forward to $Z$. Therefore we, we can assume that pull-back exists with respect to $Z$.

First we deal with the $HVC$. When the component of $HVC$ is of the form 
$$n_1p_1 \times E_2 \times E_3 \times E_4 + n_2E_1 \times p_2 \times E_3 \times E_4 + n_3E_1 \times E_2 \times p_3 \times E_4$$

$\newline$then that pulls back to $\sum_i m_{1,i} \cdot q_{1,i} \times E_4$ on $Z$, because $E_4$ doesn't change in the component above, as points vary in $E_1$, $E_2$, or $E_3$. It follows, we are pulling back from $X$ to $Pr_{1,2,3}(A)$, which is a curve, and the divisors of curves are points. 

When we pull-back the component

$$n_4E_1 \times E_2 \times E_3 \times p_4$$

Then, as $p_4$ is a point in $E_4$, and since $E_1 \times E_2 \times E_3$ contains $Pr_{1,2,3}(A)$, so the pull-back will be $n_4 Pr_{1,2,3}(A) \times p_4$. It follows the pull-back of that component will be of the form, $\sum_i m_{2,i} \cdot Pr_{1,2,3}(A) \times q_{2,i}$. Therefore the pull-back of $HVC$ will be 

$$\sum_i m_{1,i} \cdot q_{1,i} \times E_4 + m_{2,i} \cdot Pr_{1,2,3}(A) \times q_{2,i}$$

Since $X$ is smooth, there exists triple $(\lambda,X,M)$ such that $\div(\lambda)_X = H_3 \cap X - HVC$. As, we have computed above, the pull-back of the effective Cartier divisor $HVC$ to $Z$ with respect to the inclusion map $Z \rightarrow X$, will be of the form of

 $$\sum_i m_{1,i} \cdot q_{1,i} \times E_4 + m_{2,i} \cdot Pr_{1,2,3}(A) \times q_{2,i}.$$

Next we will prove that the pull-back of $H_3 \cap X$ on $Z = Pr_{1,2,3}(A) \times E_4$ will be $A$. Thus the pull-back of $(\lambda,X,M)$ will be the triple $(\beta,Z,J)$ such that $\div(\beta)_Z = A - HVC$.

We will need the following result; for proof, see \cite{JP}, Theoreme 6.3, part (4).

\begin{thm}
If X is an algebraic variety of dimension at least 2 and H is a generic hyperplane, then $H \cap X$ is irreducible.
\end{thm}

\begin{lem}
Over the surface $Z = Pr_{1,2,3}(A) \times E_4$, the divisor $D_A = H_3 \cap X$ pulls-back to $A$.
\end{lem}

\begin{proof}
$dim(Z) \geq 2$, $H_3$ is a hyperplane section, and it follows by theorem 8.3 that $Z \cap H_3$ is irreducible. Since $Z$ contains $A$, and $H_3 \cap X$ contains $A$, so $Z \cap H_3$ contains $A$. Since $Z \cap H_3$ is irreducible, we have $Z \cap H_3 = A$.
\end{proof}

Now going back to Proposition $8.2$. As $H_3 \cap X$ pull back to $A = H_1 \cap H_2 \cap H_3 \cap X$, and we have proven $HVC$ pulls back to HVC, so it follows the pull back of the triple $(\lambda,X,M)$ on $Z$ will be the triple $(\beta,Z,J)$ satisfying $\div(\beta)_Z = H_1 \cap H_2 \cap H_3 \cap X - HVC$. Thus proving Proposition 8.2.
\end{proof}

 Now we will prove Lemma 8.1.
 
 \begin{proof}
 
 Set $\overline{\zeta}_1 = (\sigma,S,L) - (\beta,Z,J)$, then as $\eta \restriction Z = 0$, so it follows 
 
 $$\underline{r}_{3,1}(\overline{\zeta}_1)(\eta) = \int_{S} \log \norm*{\sigma} \eta$$
 
 As $\div(\overline{\zeta}_1) = (A - HVC) - (A - HVC) = \pm HVC$, so it follows, we have shown the following properties:
 
\begin{enumerate}
	\item $\displaystyle \underline{r}_{3,1}(\overline{\zeta}_1)(\eta) = \int_{S} \log \norm*{\sigma} \eta$.
	\item $\div(\overline{\zeta}_1) = HVC$.
\end{enumerate}

\end{proof}

In Lemma 8.1, we choose a homological HVC representative of the curve $A$, which gives us the $HVC$ on $S$. Let us deal with a different representative. We will be working with regular higher Chow pre-cycles on $S$ that we correct by a twisted precycles. 

Recall higher Chow precycles are precycles of the form defined in 5.4, but instead of rational sections over a non-trivial line bundle, the rational sections are over the trivial line bundle. In other words, the rational sections are just rational functions over subvarieties. 

We know that $H_3 \cap X - H_4 \cap X \sim_{rat,X} 0$, where $H_4 \cap X$, is a hyperplane section, as all hyperplanes are rationally equivalent. Therefore this restricts to the following cycle $(f,S)$ on $S$, where $\div(f)_S = H_1 \cap H_2 \cap H_3 \cap X - H_1 \cap H_2 \cap H_4 \cap X = A - B$.

\begin{prop}
	The precycle $(f,S)$ can be completed to a precycle $\overline{\zeta}_2$ such that 
	\begin{enumerate}
		\item $\div(\overline{\zeta}_2) = HVC$.
		\item $\displaystyle \underline{r}_{3,1}(\overline{\zeta}_2)(\eta) =  \int_{S} \log \norm*{\sigma'} \eta$.
	\end{enumerate}
\end{prop}

\begin{prop}
Over the surfaces $Z_1 = Pr_{1,2,3}(A) \times E_4$, and $Z_2 = Pr_{1,2,3}(B) \times E_4$, there exists triple $(\beta_1,Z,J_1)$ and $(\beta_2,Z,J_2)$ such that 

\begin{enumerate}
	\item $\div(\beta_1)_{Z_1} = A - HVC$;
	\item $\div(\beta_2)_{Z_2} = B - HVC$. 
\end{enumerate}	
\end{prop}

\begin{proof}
Let $Z_1 = Pr_{1,2,3}(A) \times E_4$. From Proposition 8.2, we can construct the precycle $(\beta_1,Z_1,J_1)$ where $\div(\beta_1)_{Z_1} =A - HVC$. Set $Z_2 = Pr_{1,2,3}(B) \times E_4$, by a similar argument as Proposition 8.2, we can construct triple $(\beta_2,Z_2,M_2)$ such that $\div(\beta_2)_{Z_2} = B - HVC$. 

Hence if we set the precycle 

\begin{align}
	\overline{\zeta_2} = (f, S) - (\beta_1,Z_1,J_1) + (\beta_2,Z_2,J_2)
\end{align}

Then if follows that this precycle satisfies the following

\begin{align}
	\div(\overline{\zeta_2}) &= HVC \\
	\underline{r}_{3,1}(\overline{\zeta}_2)(\eta) &= \int_{S} \log \norm*{f} \eta
\end{align}
	
\end{proof}

\end{section}

\begin{section}{Main result}

Finally we record everything as the following theorem that we get as a corollary of our above construction. This tells us that any precycle $\sum_i (\sigma_i,S,L_i)$ that is induced from hyperplane section can be completed to twisted cycle, where $\div(\sigma_i)_S = H_i \cap S - HVC$ or $\div(\sigma_i)_S = H_i \cap S - H_j \cap S$.

\begin{thm}
Any precycle $\sum_i (\sigma_i,S,L_i)$ that is induced from hyperplane sections on $S$ can be corrected to twisted cycle $\zeta$ such that 
\begin{align}
	\underline{r}_{3,1}(\zeta)(\eta) =  \int_{S} \log \norm*{\sigma_i} \eta
\end{align}

\end{thm}

\begin{proof}
By the hyperplane arrangements, we have there exists $\overline{\zeta}$ such that $\div(\overline{\zeta})$ are horizontal and vertical cycle, and moreover, $\displaystyle \underline{r}_{3,1}(\overline{\zeta})(\eta) = \sum \int_{S} \log \norm*{\sigma_i} \eta$. Since the divisor of the precycle $\overline{\zeta}$ are horizontal and vertical cycles, so we can correct the precycle $\overline{\zeta}$ using surfaces that don't support $(2,2)$ form, so the completion of $\overline{\zeta}$, which is $\zeta$ will satisfy $(5)$. For completeness sake, let us construct the twisted cycle explicitly. 

We will assume that the harder case, where we are dealing with rational function $f$ on $S$ such that $\div(f)_S = A - B$, where $A = H_3 \cap S = H_1 \cap H_2 \cap H_3 \cap X$, and $B = H_1 \cap H_2 \cap H_4 \cap X$, where $H_i$ are hyperplane sections (The other case, where $B = HVC$ can be dealt with using a similar argument).

By Proposition 8.2, if we set $\overline{\zeta}_2 = (f,S) - (\beta_1,Z_1,J_1) + (\beta_2,Z_2,J_2)$, then we have that $\div(\overline{\zeta}_2) = HVC$, where $Z_1 = Pr_{1,2,3}(A) \times E_4$, and $Z_2 = Pr_{1,2,3}(B) \times E_4$. Let us explicitly compute the $HVC$. First of all we have the following equations

\begin{align}
	\div(\beta_1)_{Z_1} &= A - \sum m_{1,i} a_{1,i} \times a_{2,i} \times a_{3,i} \times E_4 + m_{2,i} Pr_{1,2,3}(A) \times a_{4,i} \\
	\div(\beta_2)_{Z_2} &= B - \sum m_{1,i} b_{1,i} \times b_{2,i} \times b_{3,i} \times E_4 + m_{2,i} Pr_{1,2,3}(A) \times b_{4,i}
\end{align}

We have assumed that the $HVC$ over $Z_1$ and $Z_2$ have same multiplicity just for simplicity, otherwise we could just multiply everything out, so wlog we can assume they have the same multiplicity. 

From $(6)$ and $(7)$, we have that 

\begin{align}
	\div(\overline{\zeta}_2) = &\sum_i n_{1,i} \cdot \biggl(a_{1,i} \times a_{2,i} \times a_{3,i} \times E_4 - b_{1,i} \times b_{2,i} \times b_{3,i} \times E_4 \biggl) + \\
	&\sum_i n_{2,i} \cdot \biggl(Pr_{1,2,3}(A) \times a_{4,i} - Pr_{1,2,3}(B) \times b_{4,i} \biggl)
\end{align}

We can can also assume that $n_i = n_{1,i} = n_{2,i}$, as we can apply the argument below to each multiplicity. First we will correct the component given in $(8)$, by an argument invoking Bertini theorem, we can find a smooth curve $C_i \subset E_1 \times E_2 \times E_3$ containing the points $a_i = a_{1,i} \times a_{2,i} \times a_{3,i}$ and $b_i = b_{1,i} \times b_{2,i} \times b_{3,i}$. All points over a curve are homologically equivalent, and since $C_i$ is smooth, so it follows that $a_i \sim_{hom,C_i} b_i$. Therefore there exists triple $(\gamma_i,C_i,R_i)$ such that $\div(\gamma_i)_{C_i} = a_i - b_i$, from which it follows that 

\begin{align}
	\div(\gamma_i^{n_i})_{C_i \times E_4} &= n_i \cdot \biggl( a_i \times E_4 - b_i \times E_4 \biggl) \\ 
	&= n_i \cdot \biggl(  a_{1,i} \times a_{2,i} \times a_{3,i} \times E_4 - b_{1,i} \times b_{2,i} \times b_{3,i} \times E_4 \biggl)
\end{align}

If we set $\zeta_2 = (f,S) - (\beta_1,Z_1,J_1) + (\beta_2,Z_2,J_2) - \sum (\gamma_i^{n_i},C_i \times E_4,R_i)$, first we note $C_i \times E_4$ doesn't support $(2,2)$ real harmonic form $\eta$, that is we have that $\eta \restriction C_i \times E_4 = 0$. Now if we calculate the divisor of $\zeta_2$, we have the following  

\begin{align}
	\div(\zeta_2) = \sum_i n_i \cdot \biggl(Pr_{1,2,3}(A) \times a_{4,i} - Pr_{1,2,3}(B) \times b_{4,i} \biggl)
\end{align}

We will correct $(12)$ following two step process. First, we have that over $E_4$ we have $a_{4,i} \sim_{hom,E_4} b_{4,i}$. Since $E_4$ is smooth, so there exists triple $(\nu_i,E_4,F_i)$, such that $\div(\nu_i)_{E_4} = a_{4,i} - b_{4,i}$. Therefore 

\begin{align}
	\div(\nu_i^{n_i})_{Pr_{1,2,3}(A) \times E_4} = n_i \cdot \biggl(Pr_{1,2,3}(A) \times a_{4,i} - Pr_{1,2,3} \times b_{4,i} \biggl)
\end{align}

Therefore set $\zeta_1 = \zeta_2 - \sum_i (\nu_i^{n_i},Pr_{1,2,3}(A) \times E_4,F_i)$, again we notice that $\eta \restriction Pr_{1,2,3}(A) \times E_4 = 0$, and from $(13)$ we have the following computation 

\begin{align}
\div(\zeta_1) = \sum_{i} n_i \cdot \biggl( Pr_{1,2,3}(A) \times b_{4,i} - Pr_{1,2,3}(B) \times b_{4,i} \biggl)
\end{align}

We are now ready to correct $(14)$ to zero, first of all we know that over $X$, we have $H_3 \cap X \sim_{rat,X} H_4 \cap X$, as any two hyperplanes are rationally equivalent, from which we get that $A \sim_{rat,S} B$, and from the push forward we have 

\begin{align*}
	Pr_{1,2,3}(A) &\sim_{rat,Pr_{1,2,3}(S)} Pr_{1,2,3}(B) \implies \\
	&Pr_{1,2,3}(A) \times b_{4,i} \sim_{rat,Pr_{1,2,3}(S) \times b_{4,i}} Pr_{1,2,3}(B) \times b_{4,i}.
\end{align*}

Therefore there exists $(g_i,Pr_{1,2,3}(S) \times b_{4,i})$ such that $div(g_i)_{Pr_{1,2,3}(S) \times b_{4,i}} = Pr_{1,2,3}(A) \times b_{4,i} - Pr_{1,2,3}(B) \times b_{4,i}$. Since $b_{4,i}$ are points, so $\eta \restriction Pr_{1,2,3}(S) \times b_{4,i} = 0$.

$\newline$

$\newline$Finally, if we set $\zeta$ to be defined as follows
\begin{align}
\zeta &= (f,S) - (\beta_1,Z_1,J_1) + (\beta_2,Z_2,J_2) - \sum_i (\gamma_i^{n_i},C_i \times E_4,R_i)\\
&- \sum_i (\nu_i^{n_i},Pr_{1,2,3}(A) \times E_4,F_i) - \sum_i (g_i,Pr_{1,2,3}(S) \times b_{4,i})
\end{align}

We have $\div(\zeta) = 0$, and we also have all surfaces in $\zeta$ doesn't support $(2,2)$ real form except $S$, so it follows that 

$$r_{3,1}(\zeta)(\eta) = \int_{S} \log \abs{f} \eta$$

By extending this argument to multiple precycles, we establish the theorem. 
\end{proof}

\begin{cor}
Suppose that we have a precycle $(\sigma,S,L)$ such that $\div(\sigma)_S = A - B$, such that the cycles $A$ and $B$ satisfy:

\begin{itemize}
	\item $A$ is a complete intersection.
	\item $B$ is a complete intersection or HVC.
\end{itemize}

$\newline$
Further, suppose that there exists $(2,2)$ real harmonic form $\eta$, where the regulator is non-trivial, i.e. 

$$\int_{S} \log \norm*{\sigma} \eta \neq 0.$$

$\newline$Then we can complete the triple $(\sigma,S,L)$ to twisted cycle $\zeta$, where $\zeta$ satisfies 

$$\underline{r}_{3,1}(\zeta)(\eta) = \int_{S} \log \norm*{\sigma} \eta \neq 0.$$

\end{cor}

\begin{rem}
We note the following remark the idea of constructing the indecomposable relies on the following facts:

	\begin{itemize}
		\item Complete intersections behave well homologically, or you could think linear algebraically.
		\item They move easily.
	\end{itemize}
\end{rem}

\begin{rem}
For future work, we will generalize this idea to the product of elliptic curves, construct indecomposables, generalize to the product of curves, and use that to prove the Twisted Hodge D conjecture and establish the strong version of this conjecture, as well as the strong version of this conjecture for smooth projective surfaces.
\end{rem}

\end{section}


\begin{thebibliography}{100000000}


\bibitem[BT]{BT} H. Bass and J. Tate, The Milnor ring of a global
field, in {\it Algebraic $K$-theory II,} Lecture Notes in Math. 342,
Springer-Verlag, 1972, 349--446.

\bibitem[Bei]{Bei} A. Beilinson, Higher regulators and values of
$L$-functions, J. Soviet Math. {\bf 30}, 1985, 2036--2070.

\bibitem[Blo1]{Blo1} S. Bloch,  Algebraic cycles and higher $K$-theory,
Adv. Math. {\bf 61}, 1986, 267--304.

\bibitem[Blo2]{Blo2} \vrule height1pt width35pt depth0pt, 
Algebraic cycles and the Beilinson conjectures, Contemporary
Mathematics Vol. {\bf 58}, Part I, 1986, 65--79.

\bibitem[Blo3]{Blo3} \vrule height1pt width35pt depth0pt,  
{\it Lectures on Algebraic Cycles,}\ Duke University 
Mathematics Series IV, 1980.

\bibitem[C-L1]{C-L1} X. Chen and J. D. Lewis, Noether-Lefschetz
for $K_{1}$ of a certain class of surfaces,
Bol. Soc. Mexicana (3) Vol. 10, 2004.

\bibitem[C-L2]{C-L2}  \vrule height1pt width50pt depth0pt,  
The Hodge-${\cal D}$-conjecture
for $K_{3}$ and Abelian surfaces, J. Alg. Geometry {\bf 14},
2005, 213--240.

\bibitem[Co]{Co} A. Collino, Griffiths' infinitesimal invariant and
higher $K$-theory on hyperelliptic jacobians, J. Algebraic
Geometry {\bf 6}, 1997, 393--415.

\bibitem[E-M]{E-M} P. Elbaz-Vincent and S. M\"uller-Stach, Milnor
$K$-theory, higher Chow groups and applications, Invent. math. {\bf 
148}, 2002, 177--206.

\bibitem[EV]{EV} H. Esnault and E. Viehweg, Deligne-Beilinson cohomology,
in {\it Beilinson's Conjectures on Special Values of $L$-Functions,}
(Rapoport, Schappacher, Schneider, eds.), Perspectives in Math. 4,
Academic Press, San Diego, 1988, 43--91.

\bibitem[GL]{GL} B. Gordon and J. Lewis; Collaboration.

\bibitem[GH]{GH} P. Griffiths and J. Harris, {\it Principles of Algebraic 
Geometry,} John Wiley \& Sons, New York, 1978.

\bibitem[Ja]{Ja} U. Jannsen, Deligne cohomology, Hodge-${\cal D}$-conjecture,
and motives, in {\it Beilinson's Conjectures on Special Values of
$L$-Functions,} (Rapoport, Schappacher, Schneider, eds.), Perspectives in Math.
4, Academic Press, San Diego, 1988, 305--372.
\bibitem[JP]{JP} Jouanolou, J. -P. (1983). Theoremes de Bertini et applications. Birkhauser.

\bibitem[Ka]{Ka} K. Kato, Milnor $K$-theory and the Chow group of zero
cycles, in {\it Applications of $K$-theory to Algebraic Geometry and
Number Theory, Part I,} Contemp. Math. 55, 1986, 241--253.

\bibitem[Ke1]{Ke1} M. Kerr, Geometric construction of regulator currents
with applications  to algebraic cycles, Princeton University Thesis, 
2003.

\bibitem[Ke2]{Ke2} \vrule height1pt width35pt depth0pt,  
A regulator formula for Milnor $K$-groups.  
$K$-Theory {\bf 29} (2003), no. 3, 175--210.

\bibitem[Lev]{Lev} M. Levine, Localization on singular varieties,
Invent. Math. {\bf 31}, 1988, 423--464.

\bibitem[Lew1]{Lew1} J. Lewis, Real regulators on Milnor complexes,
$K$-Theory {\bf 25}, 2002, 277--298.

\bibitem[Lew2]{Lew2}  \vrule height1pt width35pt depth0pt, 
Regulators of Chow cycles on
Calabi-Yau varieties, in {\it Calabi-Yau Varieties and
Mirror Symmetry,}\ (N. Yui, J. D. Lewis, eds.), Fields
Institute Communications {\bf 38}, (2003), 87--117.


\bibitem[Lew3]{Lew3} J. Lewis, Real regulators on Milnor complexes II,
Algebraic Cycles and Motives: Volume 2 {\bf 25}, 2007, 214--240.


\bibitem[LR]{LR} Lazarsfeld, R. (2004). Positivity in algebraic geometry I. Classical setting: Line bundles and Linear Series. Springer. 
\bibitem[MD]{MD} Mumford, D. (1995). Algebraic geometry I complex projective varieties. Springer-Verlag. 

\bibitem[MS1]{MS1} S. M\"uller-Stach, Constructing indecomposable motivic
cohomology classes on algebraic surfaces, J. Algebraic Geometry {\bf 6},
1997, 513--543.

\bibitem[MS2]{MS2} \vrule height1pt width60pt depth0pt, Algebraic cycle
complexes, in {\it Arithmetic and Geometry of Algebraic Cycles,}\ (Gordon,
Lewis, M\"uller-Stach, S.~Saito, Yui, eds.), Kluwer Academic Publishers,
Dordrecht, The Netherlands, 2000, 285--305.

\bibitem[MS3]{MS3} \vrule height1pt width60pt depth0pt, A remark on height 
pairings, in {\it Algebraic Cycles and Hodge
Theory, Torino, 1993,}\ (A. Albano, F. Bardelli, eds.),
Lecture Notes in Mathematics {\bf 1594}, Springer-Verlag (1994), 
253--259.


\bibitem[Sou]{Sou} C. Soul\'e, {\it Lectures on Arakelov 
Geometry,} Cambridge
Studies in Advanced Mathematics 33, Cambridge University Press,
Cambridge, England, 1992.


\end{thebibliography}
\end{document}